%% file: paper.tex
\documentclass[11pt]{article}

\usepackage{amsmath,amssymb,amsfonts,textcomp,amsthm,xifthen,psfrag,graphicx,color,MnSymbol}
\usepackage[T1]{fontenc}

\usepackage{hyperref}

\usepackage{pgfplots}
\pgfplotsset{compat=newest}
\usepgfplotslibrary{external}
\date{\today}

\input{header}

\title{An ultraweak formulation of the Reissner--Mindlin plate bending model and DPG approximation
\thanks{Supported by CONICYT through FONDECYT projects 1190009, 11170050, and by NSF through grant DMS-1818867}}
\author{
Thomas~F\"uhrer$^\dagger$
\and
Norbert Heuer\thanks{
Facultad de Matem\'aticas, Pontificia Universidad Cat\'olica de Chile,
Avenida Vicu\~na Mackenna 4860, Santiago, Chile,
email: {\tt \{tofuhrer,nheuer\}@mat.uc.cl}}
\and
Francisco-Javier Sayas\thanks{
Department of Mathematical Sciences, University of Delaware, Newark DE 19716}}

\begin{document}
\date{Dedicated to our dear friend Francisco ``Pancho'' Javier Sayas\\ who passed away
in April 2019.}
\maketitle
\begin{abstract}

We develop and analyze an ultraweak variational formulation of the Reissner--Mindlin
plate bending model both for the clamped and the soft simply supported cases.
We prove well-posedness of the formulation, uniformly with respect to the plate thickness $t$.
We also prove weak convergence of the Reissner--Mindlin solution to the solution
of the corresponding Kirchhoff--Love model when $t\to 0$.

Based on the ultraweak formulation, we introduce a discretization of the discontinuous
Petrov--Galerkin type with optimal test functions (DPG) and prove its uniform
quasi-optimal convergence. Our theory covers the case of non-convex polygonal plates.

A numerical experiment for some smooth model solutions with fixed load
confirms that our scheme is locking free.


\end{abstract}

\section{Introduction}

We develop a uniformly well-posed ultraweak formulation of the Reissner--Mindlin plate bending
model and, based on this formulation, define a discontinuous Petrov--Galerkin method with
optimal test functions (DPG method) for its approximation. The objective of this work is to
continue to develop DPG techniques for plate bending models, without assuming unrealistic
regularity of solutions. The DPG framework has been proposed by Demkowicz and Gopalakrishnan
\cite{DemkowiczG_11_CDP} with the aim to automatically satisfy discrete inf-sup conditions of discretizations.
Without going into the details of advantages and challenges here, we consider this framework as
a means to give full flexibility in the design and selection of a variational formulation.
In other words, for a given problem, one can select any set of variables of interest.
The only challenge is to develop a well-posed formulation that gives access to these variables.
Then, a conforming discretization will be automatically quasi-optimal. Furthermore, it is robust
(constants do not depend on singular perturbation parameters) if the formulation is uniformly
well posed. This result assumes that one uses so-called optimal test functions,
see~\cite{DemkowiczG_11_CDP}, or approximated test functions of spaces for
which (uniformly bounded) Fortin operators exist, cf.~\cite{GopalakrishnanQ_14_APD}.

In this paper we focus on the continuous setting of the Reissner--Mindlin model.
In \cite{FuehrerHN_19_UFK} we considered clamped plates of the
Kirchhoff--Love model and afterwards, in \cite{FuehrerH_19_FDD}, provided a fully discrete analysis.
We also studied the pure deflection case \cite{FuehrerHH_TOB}, that is, the bi-Laplacian,
developing a thorough continuous analysis and giving initial results for its discretization.
In this paper, we extend the formulation and method for the clamped Kirchhoff plate from
\cite{FuehrerHN_19_UFK}.

It is well known that the Reissner--Mindlin model transforms
in a singularly perturbed way into the Kirchhoff--Love model when the plate thickness $t\to 0$.
For a plate with smooth boundary, Arnold and Falk have shown the strong convergence
of the Reissner--Mindlin deflection and rotation to the Kirchhoff--Love deflection and gradient
of the deflection when $t\to 0$. They proved convergence for different
boundary conditions and a whole scale of Sobolev norms, depending on the regularity of the solution.
Babu\v{s}ka and Pitk\"aranta \cite{BabuskaP_90_PPH} discuss the case of convex polygonal plates.
We do not know of any strong convergence result in cases of lowest regularity and non-smooth
boundary, specifically not for non-convex polygons.
In contrast, a justification of both models for small $t$ is a different subject, and has been
studied, e.g., by Arnold \emph{et al.} \cite{ArnoldMZ_02_RAR} and Braess \emph{et al.}
\cite{BraessSS_11_JPM}, the latter paper including the case of non-convex polygonal plates.

In \cite{FuehrerHN_19_UFK}, we presented a bending-moment formulation (unknowns are the vertical
deflection and the bending moment tensor). In order to extend this formulation we therefore
aim at a bending-moment based formulation of the Reissner--Mindlin model
(for clamped and soft simply-supported plates) that transforms into the Kirchhoff--Love
formulation when $t\to 0$. Specifically, the ultraweak formulation should be well posed uniformly in $t$
and the DPG approximation should be uniformly quasi-optimal (locking free).
This is exactly what we are going to achieve
at an abstract level, including the weak convergence of the Reissner--Mindlin solution to
the Kirchhoff--Love solution.
The construction of appropriate approximation spaces that guarantee this behavior for low-regular
cases (including non-convex polygonal plates) is an open problem, as is the construction of
related Fortin operators.
Furthermore, it is by no means obvious that our selection of variables (based on the objective
to extend our Kirchhoff--Love formulation) is the most convenient when it comes to constructing
approximation spaces. Considering alternative variables and formulations will be the subject of
future research.

As in \cite{FuehrerHN_19_UFK}, our focus is to develop a formulation that requires minimum
regularity, only subject to the $L_2$-regularity of the vertical load. This condition is owed
to the discontinuity of test functions of DPG schemes.
Ultraweak formulations are obtained by integrating by parts as often as necessary
to remove all derivatives from the unknown functions. This automatically generates trace operations,
and the involved traces have to be considered as independent unknowns. In other words, studying
ultraweak formulations based on minimal regularity is equivalent to studying related trace operations
and their well-posedness subject to minimal regularity requirements.
These trace operators and their jumps precisely characterize
conformity of the underlying spaces of minimum regularity and of their (conforming)
approximations. Therefore, this part of our analysis is relevant independently of the DPG scheme
we propose. Here, we consider domains with Lipschitz boundary (thus, including polygonal non-convex cases)
and notice that our analysis applies to two and three dimensions.


It goes without saying that the Reissner--Mindlin model is relevant in structural mechanics until today.
Correspondingly, there is vast literature both in mathematics and engineering sciences,
and we do not intend to discuss it to any length here. A key point in the numerical analysis has been the
locking effect that causes some numerical schemes to behave badly when $t$ becomes small.
Our scheme, being well behaved uniformly in $t$, is locking free (when using optimal test functions)
with respect to the variables of interest,
like several other known schemes. For instance, to give some mathematical references,
Stenberg and co-authors have derived locking-free schemes, e.g. \cite{ChapelleS_98_OLO},
Beir\~{a}o da Veiga \emph{et al.} \cite{BeiraoDaVeigaMR_13_NAL}
present a locking-free mixed scheme that includes the bending moment as an unknown.
In \cite{BoesingC_15_WOP}, B\"osing and Carstensen prove that a (weakly over-penalized)
discontinuous Bubnov--Galerkin method approximating the deflection and rotation variables is locking free.
We also note that there are two contributions on the DPG method for thin body problems.
Niemi \emph{et al.} \cite{NiemiBD_11_DPG} obtained a robust DPG approximation for a particular
Timoshenko beam problem, and Calo \emph{et al.} \cite{CaloCN_14_ADP} propose and analyze
a DPG scheme for the Reissner--Mindlin model, though ignoring the dependence of estimates on $t$.

An overview of the remainder of this paper is as follows.
In the next section we introduce and discuss our model problem, and make an initial step towards
a variational formulation. Section~\ref{sec_traces_jumps} is devoted to the spaces,
norms, and trace operations that are needed to formulate a well-posed ultraweak formulation.
Initially, the case $t>0$ is considered. The Kirchhoff--Love case $t=0$ is analyzed
in \S\ref{sec_jumps_KLove}. There, we recall some spaces, trace operators and
results from \cite{FuehrerHN_19_UFK}. Furthermore, we derive additional results needed
for the case of a simply supported plate, not considered in \cite{FuehrerHN_19_UFK}.
In Section~\ref{sec_uw} we then finish to develop the ultraweak formulation,
state its uniform well-posedness and weak convergence when $t\to 0$
(Theorems~\ref{thm_stab} and~\ref{thm_limit}, respectively),
define the DPG scheme, and state its robust convergence (Theorem~\ref{thm_DPG}).
Proofs of the theorems are given in Section~\ref{sec_adj}.
Finally, in Section~\ref{sec_num} we present some numerical results for the case
of some smooth solutions with fixed load and different values of $t$.

Throughout the paper, $a\lesssim b$ means that $a\le cb$ with a generic constant $c>0$ that is independent of
the plate thickness $t$ and the underlying mesh. Similarly, we use the notation $a\gtrsim b$.

\section{Model problem} \label{sec_model}

Let $\Omega\subset\R^\di$ (initially, $\di=2$) be a bounded simply connected Lipschitz domain
with boundary $\Gamma=\partial\Omega$. We are considering a model problem that, in two dimensions,
is the Reissner--Mindlin plate bending model with linearly elastic, homogeneous and isotropic material,
described by the relations
\begin{subequations} \label{RM1}
\begin{align}
   \label{RM1a}
   \bq &= \kappa Gt(\grad u-\bpsi),\\
   \label{RM1b}
   \MM &= -Dt^3[\nu\mathrm{tr}(\Grad\bpsi)\II+(1-\nu)\Grad(\bpsi)]
\end{align}
\end{subequations}
and the equilibrium equations
\begin{subequations} \label{RM2}
\begin{align}
   \label{RM2a}
   -\div\bq&=f,\\
   \label{RM2b}
   \bq&=\Div\MM
\end{align}
\end{subequations}
on $\Omega$. Here, $\Omega$ is the mid-surface of the plate with thickness $t>0$, $f$ the transversal
bending load, $u$ the transverse deflection, $\bpsi$ the rotation vector, $\bq$ the shear force vector,
$\MM$ the bending moment tensor, $\II$ the identity tensor,
and $\Grad$ the symmetric gradient, $\Grad\bpsi:=\frac 12(\grad\bpsi+(\grad\bpsi)^\transp)$.
Furthermore, $\nu\in (-1,1/2]$ is the Poisson ratio, $\kappa>0$ the shear correction factor, and
\[
   G=\frac E{2(1+\nu)},\quad D=\frac E{12(1-\nu^2)}
\]
with the Young modulus $E>0$. The operator $\div$ is the standard divergence, and
$\Div$ is the row-wise divergence when writing second-order tensors as $\di\times\di$ matrix functions.

Relation \eqref{RM1b} between $\MM$ and $\bpsi$ can be written like
\begin{equation} \label{RM1bb}
   \MM=-t^3\cC\Grad\bpsi
\end{equation}
with positive definite tensor $\cC$ that is independent of $t$.
We will consider a formulation depending on the two variables
$\MM$ and $u$. It is obtained by replacing $\bq$ in \eqref{RM2a} and \eqref{RM1a} through \eqref{RM2b},
and replacing $\bpsi$ in \eqref{RM1bb} through relation \eqref{RM1a} after elimination of $\bq$.
This yields the system
\begin{align*}
   -\div\Div\MM = f,\quad \MM = -t^3\cC\Grad(\grad u - \frac 1{\kappa G t}\Div\MM).
\end{align*}
The dependence of the problem on $\kappa G$ is not critical, for ease of presentation we select $\kappa G=1$.
Then, rescaling $f\to t^3 f$ and  $\MM\to t^3\MM$ we obtain
\begin{align*}
   -\div\Div\MM = f,\quad \MM = -\cC\Grad(\grad u - t^2\Div\MM).
\end{align*}
Considering a clamped plate, the boundary conditions are $u=0$ and $\bpsi=0$ on $\Gamma$,
the latter being transformed into $\grad u-t^2\Div\MM=0$ on $\Gamma$.
We also consider a (soft) simply supported plate, represented by $u=0$ and $\MM\nn=0$ on $\Gamma$.

To conclude, selecting $f\in L_2(\Omega)$ and, for ease of presentation, $t\in(0,1]$,
a strong form of our model problem is
\begin{subequations} \label{prob_tmp}
\begin{alignat}{2}
     -\div\Div\MM           &= f  && \quad\text{in} \quad \Omega\label{p1_tmp},\\
    \MM + \cC\Grad(\grad u - t^2\Div\MM) &= 0  && \quad\text{in} \quad \Omega\label{p2_tmp}\\
    \text{with}\quad
    u = 0,\quad \grad u-t^2\Div\MM &= 0 &&\quad\text{on}\quad\Gamma\label{BCc_tmp}\\
    \text{or}\qquad\qquad\qquad
    u = 0,\quad \MM\nn &= 0 &&\quad\text{on}\quad\Gamma.\label{BCs_tmp}
\end{alignat}
\end{subequations}
Furthermore, from now on, $\Omega$ is a bounded simply connected Lipschitz domain in $\R^d$ ($d\in\{2,3\}$).
We note that, setting $t=0$, this problem with boundary condition \eqref{BCc_tmp}
is the Kirchhoff--Love plate bending model in the form being studied in \cite{FuehrerHN_19_UFK}
(there, we also permitted $d\in\{2,3\}$).
Our aim is to develop for both boundary conditions uniformly well-posed ultraweak variational formulations
of \eqref{prob_tmp}, and uniformly quasi-optimal DPG schemes for bounded, non-negative
plate thickness including the Kirchhoff--Love case.

Now, in order to have a well-posed problem one has to select appropriate spaces.
Before starting to discuss their selection, let us introduce some notation.
Let $\cO\subset\Omega$ be a sub-domain.
$L_2$-spaces for scalar, vector and tensor-valued functions on $\cO$ are denoted by
$L_2(\cO)$, $\bL_2(\cO)$ and $\LL_2(\cO)$, respectively. Their $L_2$-norms are
$\|\cdot\|_\cO$, generically for the three cases. Also, we drop the index $\cO$ of the norm when $\cO=\Omega$.
The notation $\LL_2^s(\cO)$ refers to the subspace of symmetric $L_2$-tensors.
The spaces $H^1(\cO)$ and $\bH^1(\cO)$ are the standard
$H^1$-spaces of scalar and vector-valued functions with respective subspaces
$H^1_0(\cO)$ and $\bH^1_0(\cO)$ of vanishing traces on $\partial\cO$.
We also need $\Hdiv{\cO}$, denoting $\bL_2(\cO)$-elements whose divergence are elements of $L_2(\Omega)$.
Correspondingly, $\HDiv{\cO}$ consists of $\LL_2^s(\cO)$-tensors $\TTheta$ with $\Div\TTheta\in\bL_2(\cO)$,
and $\HDivz{\cO}\subset\HDiv{\cO}$ is the subspace of tensors with zero normal trace on $\partial\cO$.
Central to the analysis of the Kirchhoff--Love model \cite{FuehrerHN_19_UFK,FuehrerH_19_FDD}
is the space $\HdDiv{\cO}$. It consists of the completion of $\DDs(\bar\cO)$ (smooth
symmetric tensors with support in $\bar\cO$) with respect to the norm
\[
   \|\cdot\|_{\dDiv,\cO} := \bigl(\|\cdot\|_\cO^2 + \|\div\Div\cdot\|_\cO^2\bigr)^{1/2}.
\]
For the Kirchhoff--Love case we also need the standard spaces of scalar functions
$H^2(\cO)$ and $H^2_0(\cO)$ with norm
$\|\cdot\|_{2,\cO}:=\bigl(\|\cdot\|_\cO^2 + \|\Grad\grad\cdot\|_\cO^2\bigr)^{1/2}$.
As before, we drop the index $\cO$ when $\cO=\Omega$.

Now, returning to the discussion of \eqref{prob_tmp}, by \eqref{p1_tmp} it holds $\MM\in\HdDiv{\Omega}$.
The deflection variable $u$ will be taken in $H^1_0(\Omega)$, and the eliminated rotation variable
suggests $\bpsi=\grad u-t^2\Div\MM\in\bH^1_0(\Omega)$ (for the clamped plate) or
$\grad u-t^2\Div\MM\in\bH^1(\Omega)$ (for the simple support).
It turns out that the regularity $u\in H^1(\Omega)$ has to be added as a constraint to \eqref{prob_tmp},
it cannot be deduced from the equations.
To derive our ultraweak variational formulation it is paramount to incorporate this constraint into the
PDE system. We do this by introducing the additional variable $\btheta:=\grad u$.
Furthermore, since, in particular, $\div(\btheta-\grad u)=0$, we can incorporate this relation
into equation \eqref{p1_tmp} to obtain a skew self-adjoint problem, in the following sense.
Ultraweak formulations give rise to independent trace variables and, redundantly incorporating the relation
$\div(\btheta-\grad u)=0$ in \eqref{p1_tmp}, the corresponding trace operator is defined
by a skew symmetric bilinear form. This will simplify our notation and analysis.

Our reformulated strong form of the model problem is
\begin{subequations} \label{prob}
\begin{alignat}{2}
     -\div(\Div\MM + t(\btheta-\grad u))  &= f  && \quad\text{in} \quad \Omega\label{p1},\\
    \MM + \cC\Grad(\grad u - t^2\Div\MM) &= 0  && \quad\text{in} \quad \Omega\label{p2},\\
    \btheta-\grad u                     &=0   && \quad\text{in} \quad \Omega\label{p3},\\
    u = 0,\quad \grad u-t^2\Div\MM &= 0 &&\quad\text{on}\quad\Gamma\quad\text{or}\label{pBCc}\\
    u = 0,\quad \MM\nn &= 0 &&\quad\text{on}\quad\Gamma.\label{pBCs}
\end{alignat}
\end{subequations}
Note that, setting $t=0$, \eqref{p1}--\eqref{pBCc} turns into the Kirchhoff--Love plate bending model
whose ultraweak setting was proposed and analyzed in \cite{FuehrerHN_19_UFK}.
Though, setting $t=0$ in our variational formulation (to be developed),
we recover the Kirchhoff--Love model from \cite{FuehrerHN_19_UFK}
without $\btheta=\grad u$ as independent variable. This is due to the fact that the appropriate
weighting of \eqref{p3} is by the factor $t$, just like in \eqref{p1}.

We now start to develop an ultraweak formulation of \eqref{prob}.
Although the physically relevant case is $\di=2$ we present our analysis for $\di\in\{2,3\}$.
In order to use a DPG discretization we invoke product test spaces. These product spaces are induced
by a (family of) mesh(es) $\cT$ consisting of general non-intersecting Lipschitz elements $\{T\}$
so that $\bar\Omega=\cup\{\bar T;\; T\in\cT\}$. We also formally denote the mesh skeleton
by $\cS=\{\partial T;\;T\in\cT\}$.
Considering test functions $z\in L_2(\Omega)$, $\TTheta\in\LL_2^s(\Omega)$ (symmetric $L_2$-tensors),
and $\btau\in\bL_2(\Omega)$ ($L_2$-vector functions),
which are sufficiently smooth on every $T\in\cT$, and testing \eqref{p1} by $-z$,
\eqref{p2} by $\cCinv\TTheta$, \eqref{p3} by $t\btau$, and integrating by parts, we obtain the relation
\begin{align} \label{VFa}
   &\vdual{u}{\div(\Div\TTheta+t(\btau-\grad z))}_\cT
   +\vdual{\MM}{\cCinv\TTheta+\Grad(\grad z-t^2\Div\TTheta)}_\cT
   +t\vdual{\btheta}{\btau-\grad z}_\cT
   \nonumber\\
   &-\sum_{T\in\cT} \dual{u}{\nn\cdot(\Div\TTheta+t(\btau-\grad z))}_{\partial T}
    +\sum_{T\in\cT} \dual{\nn\cdot(\Div\MM+t(\btheta-\grad u))}{z}_{\partial T}
   \nonumber\\
   &-\sum_{T\in\cT} \dual{\MM\nn}{\grad z-t^2\Div\TTheta}_{\partial T}
    +\sum_{T\in\cT} \dual{\grad u-t^2\Div\MM}{\TTheta\nn}_{\partial T}
   =
   -\vdual{f}{z}.
\end{align}
Here, $\vdual{\cdot}{\cdot}$ denotes the $L_2$-inner product on $\Omega$ (generically for
scalar, vector-, and tensor-valued functions) and the index $\cT$ means that differential
operators are taken piecewise with respect to $\cT$.
In the following we will use the index notation also to indicate piecewise differential operators,
e.g., $\vdual{\pwgrad u}{\btau}:=\vdual{\grad u}{\btau}_\cT$.
Furthermore, $\nn$ generically denotes the unit normal vector on $\partial T$ (for $T\in\cT$)
and $\Gamma$, pointing outside $T$ and $\Omega$, respectively.
The notation $\dual{\cdot}{\cdot}_{\omega}$, and later $\dual{\cdot}{\cdot}_\Gamma$,
indicate dualities on $\omega\subset\partial T$ and $\Gamma$, respectively, with $L_2$-pivot space.

At this point, the skeleton terms in \eqref{VFa} are well defined only for sufficiently smooth solution
and test functions. Before formulating our final variational formulation we need to
define these terms for appropriate spaces and analyze their behavior. This will be done in
the following section, before returning to the model problem in Section~\ref{sec_uw}.

\section{Trace spaces and norms} \label{sec_traces_jumps}

Initially we consider the case of positive plate thickness, for convenience $t\in (0,1]$.
At the end of this section, in \S\ref{sec_jumps_KLove}, we will address the Kirchhoff--Love case $t=0$.

We start by defining local and global test and trace spaces.
For any $T\in\cT$ we consider the space $V(T,t)\subset H^1(T)\times\LL_2^s(T)\times \bL_2(T)$
which is the completion of $\cD(\bar T)\times\DDs(\bar T)\times\bD(\bar T)$ with respect to the norm
\begin{align} \label{norm_VT}
   &\|(z,\TTheta,\btau)\|_{V(T,t)}:=
   \nonumber\\
   &\quad
   \Bigl(\|z\|_T^2 + t\|\grad z\|_T^2 + \|\TTheta\|_T^2 + t\|\btau\|_T^2
      + \|\Grad(\grad z-t^2\Div\TTheta)\|_T^2 + \|\div(\Div\TTheta+t(\btau-\grad z))\|_T^2\Bigr)^{1/2}
\end{align}
with corresponding inner product $\ip{\cdot}{\cdot}_{V(T,t)}$.
Here, $\cD(\bar T)$, $\bD(\bar T)$ and $\DDs(\bar T)$ refer to the spaces of smooth
scalar, vector and symmetric tensor functions on $\bar T$, respectively.

The spaces $V(T,t)$ induce a product space $V(\cT,t)$ with respective norm and inner product denoted by
$\|\cdot\|_{V(\cT,t)}$ and $\ip{\cdot}{\cdot}_{V(\cT,t)}$. Introducing the norm
\begin{align*} 
   &\|(z,\TTheta,\btau)\|_{U(t)}:=
   \nonumber\\
   &\quad
   \Bigl(\|z\|^2 + t\|\grad z\|^2 + \|\TTheta\|^2 + t\|\btau\|^2
      + \|\Grad(\grad z-t^2\Div\TTheta)\|^2 + \|\div(\Div\TTheta+t(\btau-\grad z))\|^2\Bigr)^{1/2},
\end{align*}
we define the global space $U(t)$ as the completion of
$\cD(\bar\Omega)\times\DDs(\bar\Omega)\times\bD(\bar\Omega)$ with respect to $\|\cdot\|_{U(t)}$,
and $\U{c}(t)$ (clamped plate) and $\U{s}(t)$ (simple support) as the subspaces of
functions $(z,\TTheta,\btau)\in U(t)$ such that, respectively,
\begin{equation} \label{Uc_BC}
   z=0,\quad \grad z-t^2\Div\TTheta=0\quad\text{on}\ \Gamma
\end{equation}
and
\begin{equation} \label{Us_BC}
   z=0,\quad \TTheta\nn=0\quad\text{on}\ \Gamma.
\end{equation}
Of course, $\|(z,\TTheta,\btau)\|_{U(t)}=\|(z,\TTheta,\btau)\|_{V(\cT,t)}$ for
$(z,\TTheta,\btau)\in U(t)$. (We prefer the notation $U(t)$ instead of $V(t)$ since this space also
characterizes the solution of our problem where we generally use the letter $U$ in variational formulations.)

\begin{lemma} \label{la_reg}
Let $t>0$. If $(z,\TTheta,\btau)\in V(\cT,t)$ then
\begin{align*}
   &z\in H^1(\cT),\ \TTheta\in\HDiv{\cT},\ \pwgrad z-t^2\pwDiv\TTheta\in\bH^1(\cT),\
   \pwDiv\TTheta+t(\btau-\pwgrad z)\in\Hdiv{\cT}
\end{align*}
and, if $(z,\TTheta,\btau)\in \U{c}(t)$ or $(z,\TTheta,\btau)\in \U{s}(t)$ then, respectively,
\begin{align*}
   &z\in H^1_0(\Omega),\ \TTheta\in\HDiv{\Omega},\ \grad z-t^2\Div\TTheta\in\bH^1_0(\Omega),\
   \Div\TTheta+t(\btau-\grad z)\in\Hdiv{\Omega},
\end{align*}
or
\begin{align*}
   &z\in H^1_0(\Omega),\ \TTheta\in\HDivz{\Omega},\ \grad z-t^2\Div\TTheta\in\bH^1(\Omega),\
   \Div\TTheta+t(\btau-\grad z)\in\Hdiv{\Omega}.
\end{align*}
\end{lemma}

\begin{proof}
The stated regularities are straightforward to deduce. Note, e.g. in the case $(z,\TTheta,\btau)\in V(\cT,t)$,
that $\TTheta\in\HDiv{\cT}$ since $\grad z|_T\in \bL_2(T)$ and $(\grad z-t^2\Div\TTheta)|_T\in\bH^1(T)$
for any $T\in\cT$.
\end{proof}

\subsection{Traces} \label{sec_traces}

For $T\in\cT$, we introduce a linear operator $\traceRM{T,t}:\;V(T,t)\to (V(T,t))'$ by
\begin{align} \label{trT_def}
   \dualRM{\traceRM{T,t}(z,\TTheta,\btau)}{(\deltaz,\DeltaTheta,\deltatau)}{\partial T,t} :=\qquad&
   \nonumber\\
   \vdual{z}{\div(\Div\DeltaTheta+t(\deltatau-\grad\deltaz))}_T
   &-\vdual{\div(\Div\TTheta+t(\btau-\grad z))}{\deltaz}_T
   \nonumber\\
   +\vdual{\TTheta}{\Grad(\grad\deltaz-t^2\Div\DeltaTheta)}_T
   &-\vdual{\Grad(\grad z-t^2\Div\TTheta)}{\DeltaTheta}_T
   \nonumber\\
   -t\vdual{\btau}{\grad\deltaz}_T &+ t\vdual{\grad z}{\deltatau}_T
\end{align}
(note the additional parameter $t$ in the duality notation $\dual{\cdot}{\cdot}_{\partial T,t}$).
The range of this operator is denoted by
\[
   \HRM{}(\partial T,t) := \traceRM{T,t}(V(T,t)),\quad T\in\cT.
\]
It is easy to see that this trace operator is supported on the boundary of $T$.
Specifically, we have the following result.

\begin{lemma} \label{la_trT}
Let $t>0$. For $T\in\cT$ the trace operator $\traceRM{T,t}$ satisfies the relations
\begin{equation} \label{trT_asym}
   \dualRM{\traceRM{T,t}(z,\TTheta,\btau)}{(\deltaz,\DeltaTheta,\deltatau)}{\partial T,t}
   =
   -\dualRM{\traceRM{T,t}(\deltaz,\DeltaTheta,\deltatau)}{(z,\TTheta,\btau)}{\partial T,t}
\end{equation}
and
\begin{align} \label{trT_rep}
   \dualRM{\traceRM{T,t}(z,\TTheta,\btau)}{(\deltaz,\DeltaTheta,\deltatau)}{\partial T,t}
   =&
   \nonumber\\
   \dual{z}{\nn\cdot(\Div\DeltaTheta+t(\deltatau-\grad\deltaz))}_{\partial T}
   &-\dual{\nn\cdot(\Div\TTheta+t(\btau-\grad z))}{\deltaz}_{\partial T}
   \nonumber\\
   +\dual{\TTheta\nn}{\grad\deltaz-t^2\Div\DeltaTheta}_{\partial T}
   &-\dual{\grad z-t^2\Div\TTheta}{\DeltaTheta\nn}_{\partial T}
\end{align}
for any $(z,\TTheta,\btau), (\deltaz,\DeltaTheta,\deltatau)\in V(T,t)$.
\end{lemma}

\begin{proof}
The skew symmetry \eqref{trT_asym} is clear by definition \eqref{trT_def}.
Relation \eqref{trT_rep} follows by integration by parts subject to the required regularity
of the individual components. Let us check the regularities of the left terms of each of the
pairs appearing in \eqref{trT_rep}. By the same arguments the corresponding right terms have the required
regularities. Using the regularity provided by Lemma~\ref{la_reg}, and $T\in\cT$:
\begin{enumerate}
\item The trace of $z$ on $\partial T$ is well defined as an element of
      $H^{1/2}(\partial T)$, the trace space of $H^1(T)$.
      The normal component of $\Div\DeltaTheta+t(\deltatau-\grad\deltaz)$ on $\partial T$
      is an element of the dual space of $H^{1/2}(\partial T)$ since
      $\Div\DeltaTheta+t(\deltatau-\grad\deltaz)\in\Hdiv{T}$ by Lemma~\ref{la_reg}.
\item The trace of $\grad\deltaz-t^2\Div\DeltaTheta$ on $\partial T$ is an element
      of $\bH^{1/2}(\partial T)$, the standard trace space of $\bH^1(T)$, since
      $\grad\deltaz-t^2\Div\DeltaTheta\in\bH^1(T)$ by Lemma~\ref{la_reg}.
      Since $\TTheta\in\HDiv{T}$, also by Lemma~\ref{la_reg},
      the normal component(s) $\TTheta\nn$ on $\partial T$ is an element of
      the dual space of $\bH^{1/2}(\partial T)$.
\end{enumerate}
\end{proof}

We also introduce the corresponding collective (global) trace operator,
\[
   \traceRM{\cT,t}:\;
   \left\{\begin{array}{cll}
      U(t) & \to & V(\cT,t)',\\
      (z,\TTheta,\btau) & \mapsto & \traceRM{\cT,t}(z,\TTheta,\btau)
                                    := (\traceRM{T,t}(z,\TTheta,\btau))_{T\in\cT}
   \end{array}\right.
\]
with duality
\begin{align} \label{tr_RM}
    \dualRM{\traceRM{\cT,t}(z,\TTheta,\btau)}{(\deltaz,\DeltaTheta,\deltatau)}{\cS,t}
    := \sum_{T\in\cT} \dualRM{\traceRM{T,t}(z,\TTheta,\btau)}{(\deltaz,\DeltaTheta,\deltatau)}{\partial T,t}
\end{align}
and range $\HRM{}(\cS,t):=\traceRM{\cT,t}(U(t))$.
Here, and in the following, considering dualities $\dual{\cdot}{\cdot}_{\partial T}$
and $\dualRM{\cdot}{\cdot}{\partial T,t}$ on the whole of
$\partial T$, possibly involved traces onto $\partial T$ are always taken from $T$ without further notice
and we tacitly restrict arguments to elements $T$ where needed.

To consider the different boundary conditions we
specify the following subspaces,
\begin{align*}
   \HRM{c}(\cS,t) := \traceRM{\cT,t}(\U{c}(t))
   \quad\text{and}\quad
   \HRM{s}(\cS,t) := \traceRM{\cT,t}(\U{s}(t)).
\end{align*}
Recalling representation \eqref{trT_rep} of the local trace operator we note that a corresponding
relation holds for the global trace operator acting on $U(t)$.
 
\begin{lemma} \label{la_tr}
Let $t>0$. The trace operator $\traceRM{\cT,t}$ satisfies the relation
\begin{align*}
   \dualRM{\traceRM{\cT,t}(z,\TTheta,\btau)}{(\deltaz,\DeltaTheta,\deltatau)}{\cS,t}
   =\qquad&
   \\
   \dual{z}{\nn\cdot(\Div\DeltaTheta+t(\deltatau-\grad\deltaz))}_\Gamma
   &-\dual{\nn\cdot(\Div\TTheta+t(\btau-\grad z))}{\deltaz}_\Gamma
   \\
   +\dual{\TTheta\nn}{\grad\deltaz-t^2\Div\DeltaTheta}_\Gamma
   &-\dual{\grad z-t^2\Div\TTheta}{\DeltaTheta\nn}_\Gamma
\end{align*}
for any $(z,\TTheta,\btau), (\deltaz,\DeltaTheta,\deltatau)\in U(t)$.
\end{lemma}

\begin{proof}
The proof of this statement is analogous to that of the local variant \eqref{trT_rep}.
\end{proof}

The local and global traces are measured in the \emph{minimum energy extension} norms,
\begin{align}
   \nonumber
   \|\tq\|_\trRM{\partial T,t}
   &:=
   \inf \Bigl\{\|(z,\TTheta,\btau)\|_{V(T,t)};\;
               (z,\TTheta,\btau)\in V(T,t),\ \traceRM{T,t}(z,\TTheta,\btau)=\tq\Bigr\},\\
   \label{norm_tr_global}
   \|\tq\|_\trRM{\cS,t}
   &:=
   \inf \Bigl\{\|(z,\TTheta,\btau)\|_{U(t)};\;
               (z,\TTheta,\btau)\in \U{}(t),\ \traceRM{\cT,t}(z,\TTheta,\btau)=\tq\Bigr\}.
\end{align}
For given $\tq\in\HRM{}(\partial T,t)$ and $(\deltaz,\DeltaTheta,\deltatau)\in V(T,t)$,
we define their duality pairing by
\[
   \dualRM{\tq}{(\deltaz,\DeltaTheta,\deltatau)}{\partial T,t}
   :=
   \dualRM{\traceRM{T,t}(z,\TTheta,\btau)}{(\deltaz,\DeltaTheta,\deltatau)}{\partial T,t}
\]
where $(z,\TTheta,\btau)\in V(T,t)$ is such that $\traceRM{T,t}(z,\TTheta,\btau) = \tq$, and
\begin{align} \label{tr_RM_dual}
   \dualRM{\tq}{(\deltaz,\DeltaTheta,\deltatau)}{\cS,t}
   :=
   \sum_{T\in\cT} \dualRM{\tq_T}{(\deltaz,\DeltaTheta,\deltatau)}{\partial T,t}
\end{align}
for $\tq=(\tq_T)_{T\in\cT}\in \HRM{}(\cS,t)$ and $(\deltaz,\DeltaTheta,\deltatau)\in V(\cT,t)$.
Using these dualities we define alternative norms in the trace spaces by
\begin{align*}
   \|\tq\|_{V(T,t)'}
   &:=
   \sup_{0\not=(z,\TTheta,\btau)\in V(T,t)}
   \frac{\dualRM{\tq}{(z,\TTheta,\btau)}{\partial T,t}}{\|(z,\TTheta,\btau)\|_{V(T,t)}},
   \quad \tq\in \HRM{}(\partial T,t),\ T\in\cT,\nonumber\\
   \|\tq\|_{V(\cT,t)'}
   &:=
   \sup_{0\not=(z,\TTheta,\btau)\in V(\cT,t)}
   \frac{\dualRM{\tq}{(z,\TTheta,\btau)}{\cS,t}}{\|(z,\TTheta,\btau)\|_{V(\cT,t)}},
   \quad \tq\in \HRM{}(\cS,t).
\end{align*}

\begin{lemma} \label{la_trT_norms}
Let $t>0$. It holds the identity
\[ 
   \|\tq\|_{V(T,t)'} = \|\tq\|_\trRM{\partial T,t}\quad
   \forall \tq\in \HRM{}(\partial T,t),\ T\in \cT,
\]
so that
\[
   \traceRM{T,t}:\; V(T,t)\to \HRM{}(\partial T,t)
\]
has unit norm and $\HRM{}(\partial T,t)$ is closed.
\end{lemma}

\begin{proof}
Let $T\in\cT$ be arbitrary and fixed. By definition \eqref{trT_def} of the trace operator $\traceRM{T,t}$
and definition \eqref{norm_VT} of the norm $\|\cdot\|_{V(T,t)}$ we can bound
\begin{align*}
   \dualRM{\traceRM{T,t}(z,\TTheta,\btau)}{(\deltaz,\DeltaTheta,\deltatau)}{\partial T,t}
   &\le
   \|(z,\TTheta,\btau)\|_{V(T,t)} \|(\deltaz,\DeltaTheta,\deltatau)\|_{V(T,t)}
\end{align*}
for any $(z,\TTheta,\btau), (\deltaz,\DeltaTheta,\deltatau)\in V(T,t)$.
This proves that $\|\tq\|_{V(T,t)'}\le \|\tq\|_\trRM{\partial T,t}$.

Now let $\tq\in\HRM{}(\partial T,t)$ be given.
We define $(z,\TTheta,\btau)\in V(T,t)$ as the solution to the problem
\begin{align} \label{prob_RM_z}
   \ip{(z,\TTheta,\btau)}{(\deltaz,\DeltaTheta,\deltatau)}_{V(T,t)}
   =
   \dualRM{\tq}{(\deltaz,\DeltaTheta,\deltatau)}{\partial T,t}
   \quad\forall (\deltaz,\DeltaTheta,\deltatau)\in V(T,t).
\end{align}
We continue to define $(u,\MM,\btheta)\in V(T,t)$ as the solution to
\begin{align} \label{prob_RM_u}
   \ip{(u,\MM,\btheta)}{(\deltaz,\DeltaTheta,\deltatau)}_{V(T,t)}
   =
   \dualRM{\traceRM{T,t}(\deltaz,\DeltaTheta,\deltatau)}{(z,\TTheta,\btau)}{\partial T,t}
   \quad\forall (\deltaz,\DeltaTheta,\deltatau)\in V(T,t).
\end{align}
Selecting $(\deltaz,\DeltaTheta,\deltatau)=(z,\TTheta,\btau)$ in \eqref{prob_RM_z} shows that
\(
   \dualRM{\tq}{(z,\TTheta,\btau)}{\partial T,t} = \|(z,\TTheta,\btau)\|_{V(T,t)}^2.
\)
Now, if $(u,\MM,\btheta)$ has the trace $\tq$,
\(
   \traceRM{T,t}(u,\MM,\btheta) = \tq,
\)
then \eqref{prob_RM_u} yields
\(
   \dualRM{\tq}{(z,\TTheta,\btau)}{\partial T,t} = \|(u,\MM,\btheta)\|_{V(T,t)}^2
\)
so that
\[
   \|\tq\|_{V(T,t)'}
   \ge
   \frac{\dualRM{\tq}{(z,\TTheta,\btau)}{\partial T,t}}{\|(z,\TTheta,\btau)\|_{V(T,t)}}
   =
   \|(u,\MM,\btheta)\|_{V(T,t)}
   \ge
   \|\tq\|_\trRM{\partial T,t},
\]
which finally proves the stated norm identity.
Then we also conclude that $\HRM{}(\partial T,t)$ is closed since it is the image of a bounded below operator.

It remains to verify that
\(
   \traceRM{T,t}(u,\MM,\btheta) = \tq.
\)
We first show that
\begin{align} \label{u_z_rel}
   u &= \div(\Div\TTheta + t (\btau-\grad z)),\quad
   \MM = \Grad(\grad z - t^2 \Div\TTheta),\quad
   \btheta = -\grad z.
\end{align}
To this end we define $(\tilde u,\wilde\MM,\tilde\btheta)\in L_2(T)\times\LL_2^s(T)\times\bL_2(T)$ by
\begin{align*}
   \tilde u &:= \div(\Div\TTheta + t (\btau-\grad z)),\quad
   \wilde\MM := \Grad(\grad z - t^2 \Div\TTheta),\quad
   \tilde\btheta := -\grad z
\end{align*}
and show that it solves \eqref{prob_RM_u}. By uniqueness we then conclude that
$(\tilde u,\wilde\MM,\tilde\btheta)=(u,\MM,\btheta)$ so that \eqref{u_z_rel} holds.
Now, selecting in \eqref{prob_RM_z} smooth test functions with compact support in $T$
so that, respectively, only $\deltaz$, $\DeltaTheta$, or $\deltatau$ are non-zero, we deduce
the following relations in distributional sense,
\begin{align}
   z + \div\Bigl[\Div\Grad(\grad z-t^2\Div\TTheta)
                 - t \grad\bigl(z+\div\bigl\{\Div\TTheta + t(\btau-\grad z)\bigr\}\bigr)
           \Bigr]
   &=0,\label{z_reg}\\
   \TTheta + \Grad\Bigl[
                        \grad\div\bigl\{\Div\TTheta+t(\btau-\grad z)\bigr\}
                        - t^2\Div\Grad(\grad z-t^2\Div\TTheta)
                  \Bigr]
   &=0,\label{Theta_reg}\\
   \btau - \grad\div\bigl\{\Div\TTheta+t(\btau-\grad z)\bigr\}
   &=0.\label{tau_reg}
\end{align}
By the regularity $(z,\TTheta,\btau)\in V(T,t)$ we conclude that
\begin{alignat}{3}
   &\tilde u = \div\bigl\{\Div\TTheta + t (\btau-\grad z)\bigr\} &&\in H^1(T)
   \qquad&&\text{(from \eqref{tau_reg})},\nonumber\\
   &\Grad\bigl[\grad\tilde u-t^2\Div\wilde\MM\bigr]
   = -\TTheta &&\in \LL_2^s(T) &&\text{(from \eqref{Theta_reg})}, \label{rel_Theta}\\
   &\div\bigl(\Div\wilde\MM + t(\tilde\btheta-\grad\tilde u)\bigr) = -z &&\in L_2(T) &&\text{(from \eqref{z_reg})},
   \label{rel_z}
\end{alignat}
that is, $(\tilde u,\wilde\MM,\tilde\btheta)\in V(T,t)$. Furthermore, by definition of
$(\tilde u,\wilde\MM,\tilde\btheta)$, since $\grad\tilde u=\btau$ by \eqref{tau_reg}, and using
\eqref{rel_Theta} and \eqref{rel_z},
\begin{align*}
\lefteqn{
   \ip{(\tilde u,\wilde\MM,\tilde\btheta)}{(\deltaz,\DeltaTheta,\deltatau)}_{V(T,t)}
}
   \\&=
     \vdual{\tilde u}{\deltaz}_T
   + t\vdual{\grad\tilde u}{\grad\deltaz}_T
   + \vdual{\wilde\MM}{\DeltaTheta}_T
   + t\vdual{\tilde\btheta}{\deltatau}_T
   \\&\ + \vdual{\Grad(\grad\tilde u-t^2\Div\wilde\MM)}{\Grad(\grad\deltaz-t^2\Div\DeltaTheta)}_T
   \\&\ + \vdual{\div(\Div\wilde\MM+t(\tilde\btheta-\grad\tilde u))}
              {\div(\Div\DeltaTheta+t(\deltatau-\grad\deltaz))}_T
   \\&=
     \vdual{\div(\Div\TTheta + t (\btau-\grad z))}{\deltaz}_T
   + t\vdual{\btau}{\grad\deltaz}_T
   + \vdual{\Grad(\grad z - t^2 \Div\TTheta)}{\DeltaTheta}_T
   - t\vdual{\grad z}{\deltatau}_T
   \\&\ - \vdual{\TTheta}{\Grad(\grad\deltaz-t^2\Div\DeltaTheta)}_T
   - \vdual{z}{\div(\Div\DeltaTheta+t(\deltatau-\grad\deltaz))}_T
   \\&=
   -\dualRM{\traceRM{T,t}(z,\TTheta,\btau)}{(\deltaz,\DeltaTheta,\deltatau)}{\partial T,t}
   =
   \dualRM{\traceRM{T,t}(\deltaz,\DeltaTheta,\deltatau)}{(z,\TTheta,\btau)}{\partial T,t}.
\end{align*}
The last two relations hold by definition \eqref{trT_def} of the trace operator and
its skew symmetry \eqref{trT_asym}.
Recalling \eqref{prob_RM_u} we conclude that $(\tilde u,\wilde\MM,\tilde\btheta)=(u,\MM,\btheta)$
so that \eqref{u_z_rel} holds.
Now, using \eqref{u_z_rel}, relations \eqref{rel_Theta}, \eqref{rel_z}
with $(\tilde u,\wilde\MM,\tilde\btheta)$ replaced by $(u,\MM,\btheta)$,
and again the relation $\grad u=\grad\tilde u=\btau$, we find that
\begin{align*}
\lefteqn{
   \dualRM{\traceRM{T,t}(u,\MM,\btheta)}{(\deltaz,\DeltaTheta,\deltatau)}{\partial T,t}
}
   \\&=
   \vdual{u}{\div(\Div\DeltaTheta+t(\deltatau-\grad\deltaz))}_T
   -\vdual{\div(\Div\MM+t(\btheta-\grad u))}{\deltaz}_T
   \\&\ + \vdual{\MM}{\Grad(\grad\deltaz-t^2\Div\DeltaTheta)}_T
   -\vdual{\Grad(\grad u-t^2\Div\MM)}{\DeltaTheta}_T
   -t\vdual{\btheta}{\grad\deltaz}_T + t\vdual{\grad u}{\deltatau}_T
   \\&=
   \vdual{\div(\Div\TTheta + t (\btau-\grad z))}{\div(\Div\DeltaTheta+t(\deltatau-\grad\deltaz))}_T
   +\vdual{z}{\deltaz}_T
   \\&\ +\vdual{\Grad(\grad z-t^2\Div\TTheta)}{\Grad(\grad\deltaz-t^2\Div\DeltaTheta)}_T
   +\vdual{\TTheta}{\DeltaTheta}_T
   +t\vdual{\grad z}{\grad\deltaz}_T + t\vdual{\btau}{\deltatau}_T
   \\&=
   \ip{(z,\TTheta,\btheta)}{(\deltaz,\DeltaTheta,\deltatau)}_{V(T,t)}.
\end{align*}
Recalling \eqref{prob_RM_z} we conclude that
\[
   \dualRM{\traceRM{T,t}(u,\MM,\btheta)}{(\deltaz,\DeltaTheta,\deltatau)}{\partial T,t}
   =
   \dualRM{\tq}{(\deltaz,\DeltaTheta,\deltatau)}{\partial T,t}
   \quad\forall (\deltaz,\DeltaTheta,\deltatau)\in V(T,t),
\]
that is, $\traceRM{T,t}(u,\MM,\btheta)=\tq$. This finishes the proof.
\end{proof}

\subsection{Norm identities in the trace space} \label{sec_jumps}

\begin{prop} \label{prop_jump}
Let $t>0$. For $(z,\TTheta,\btau)\in V(\cT,t)$ and $a\in\{c,s\}$ it holds
\[
   (z,\TTheta,\btau)\in \U{a}(t) \quad\Leftrightarrow\quad
   \dualRM{\traceRM{\cT,t}(\deltaz,\DeltaTheta,\deltatau)}{(z,\TTheta,\btau)}{\cS,t} = 0
   \quad\forall (\deltaz,\DeltaTheta,\deltatau)\in \U{a}(t).
\]
\end{prop}

\begin{proof}
The direction ``$\Rightarrow$'' follows from Lemma~\ref{la_tr} by noting that,
for $(z,\TTheta,\btau), (\deltaz,\DeltaTheta,\deltatau)\in \U{c}(t)$,
the traces of $z$, $\deltaz$, $\grad\deltaz-t^2\Div\DeltaTheta$, and $\grad z-t^2\Div\TTheta$ on $\Gamma$
vanish by definition of $\U{c}(t)$. In the case of $\U{s}(t)$ we use that
the traces of $z$, $\deltaz$, $\DeltaTheta\nn$, and $\TTheta\nn$ vanish on $\Gamma$.

We prove the direction ``$\Leftarrow$''.
For brevity we denote $\U{cs}(t):=\U{c}(t)\cap\U{s}(t)$.
Let $T\in\cT$ and $(z,\TTheta,\btau)\in V(\cT,t)$ be given.
\begin{enumerate}
\item
Selecting $\deltaz=0$, $\DeltaTheta=0$, and an arbitrary $\deltatau\in\Hdiv{\Omega}$ we have
$(\deltaz,\DeltaTheta,\deltatau)\in \U{cs}(t)$ and the relation
\[
   0 = \dualRM{\traceRM{\cT,t}(\deltaz,\DeltaTheta,\deltatau)}{(z,\TTheta,\btau)}{\cS,t}
     = -t\vdual{\div\deltatau}{z} - t\vdual{\deltatau}{\grad z}_\cT.
\]
This implies that $z\in H^1_0(\Omega)$.
\item
Selecting $\deltaz=0$, $\deltatau=0$, and an arbitrary tensor $\DeltaTheta\in\DDs(\Omega)$
it follows that $(\deltaz,\DeltaTheta,\deltatau)\in \U{cs}(t)$ and, in the distributional sense,
\begin{align*}
   \Grad(\grad z-t^2\Div\TTheta)(\DeltaTheta)
   =
   \vdual{\div\Div\DeltaTheta}{z} - t^2 \vdual{\Grad\Div\DeltaTheta}{\TTheta}
   =
   \vdual{\DeltaTheta}{\Grad(\grad z-t^2\Div\TTheta)}_\cT.
\end{align*}
Here, in the last step, we used the relation
$\dualRM{\traceRM{\cT,t}(\deltaz,\DeltaTheta,\deltatau)}{(z,\TTheta,\btau)}{\cS,t} = 0$.
It follows that $\Grad(\grad z-t^2\Div\TTheta)\in\LL_2^s(\Omega)$, that is,
$\grad z-t^2\Div\TTheta\in\bH^1(\Omega)$.
\item
Selecting $\DeltaTheta=0$, $\deltatau=0$, and an arbitrary element $\deltaz\in\cD(\Omega)$, it
holds $(\deltaz,\DeltaTheta,\deltatau)\in \U{cs}(t)$ and we find, in the distributional sense, that
\begin{align*}
   \div(\Div\TTheta+t(\btau-\grad z))(\deltaz)
   &=
   \vdual{\Grad\grad\deltaz}{\TTheta} - t \vdual{\grad\deltaz}{\btau} - t \vdual{\Delta\deltaz}{z}
   \\
   &=
   \vdual{\deltaz}{\div(\Div\TTheta+t(\btau-\grad z))}_\cT.
\end{align*}
In the last step we again made use of the relation
$\dualRM{\traceRM{\cT,t}(\deltaz,\DeltaTheta,\deltatau)}{(z,\TTheta,\btau)}{\cS,t} = 0$.
We conclude that $\div(\Div\TTheta+t(\btau-\grad z))\in L_2(\Omega)$.
\item
It remains to show that the trace of $\grad z-t^2\Div\TTheta\in\bH^1(\Omega)$ on $\Gamma$ vanishes
(if $a=c$) and that the normal-normal trace of $\TTheta$ on $\Gamma$ vanishes (if $a=s$).

\begin{enumerate}
\item Case $a=c$.
For a given $\bg\in\bH^{-1/2}(\Gamma)$ (the space of normal traces of $\HDiv{\Omega}$ on $\Gamma$)
we select $\brho\in\bH^1_0(\Omega)$ such that $\vdual{\brho}{\rr}+\dual{\bg}{\rr}_\Gamma=0$
for any rigid body (plate) motion $\rr\in\ker{\Grad}$,
and define $\bpsi\in\bH^1(\Omega)/(\ker\Grad)$
as the solution to
\[
   -\Div\Grad\bpsi = \brho\quad\text{in}\ \Omega,\qquad
   \Grad(\bpsi)\nn = \bg\quad\text{on}\ \Gamma.
\]
We then select $(\deltaz,\DeltaTheta,\deltatau):=(0,\Grad\bpsi,0)$ and note that
$(\deltaz,\DeltaTheta,\deltatau)\in \U{c}(t)$. Indeed,
$\grad\deltaz-t^2\Div\DeltaTheta=-t^2\Div\Grad\bpsi=t^2\brho\in\bH^1_0(\Omega)$ and
$\div(\Div\DeltaTheta+t(\deltatau-\grad\deltaz))=-\div\brho\in L_2(\Omega)$.
Using Lemma~\ref{la_tr} and the fact that
$z$, $\deltaz$, $\grad\deltaz-t^2\Div\DeltaTheta$ have zero trace on $\Gamma$
we deduce that
\[
   0 = \dualRM{\traceRM{\cT,t}(\deltaz,\DeltaTheta,\deltatau)}{(z,\TTheta,\btau)}{\cS,t}
     = \dualRM{\traceRM{\cT,t}(0,\Grad\bpsi,0)}{(z,\TTheta,\btau)}{\cS,t}
     = \dual{\bg}{\grad z-t^2\Div\Theta}_\Gamma.
\]
Since $\bg\in\bH^{-1/2}(\Gamma)$ was arbitrary we conclude that $\grad z-t^2\Div\Theta\in\bH^1_0(\Omega)$.

\item Case $a=s$.
For a given $\bg\in\bH^{1/2}(\Gamma)$ (the trace space of $\bH^1(\Omega)$) we use an extension
$\brho\in\bH^1(\Omega)$ with $\vdual{\brho}{\rr}=0$ $\forall\rr\in\ker\Grad$,
and define $\bpsi\in\bH^1(\Omega)/(\ker\Grad)$ as the solution to
\[
   -\Div\Grad\bpsi = \brho\quad\text{in}\ \Omega,\qquad
   \Grad(\bpsi)\nn = 0\quad\text{on}\ \Gamma.
\]
We then select $(\deltaz,\DeltaTheta,\deltatau):=(0,\Grad\bpsi,0)$ and note that
$(\deltaz,\DeltaTheta,\deltatau)\in \U{s}(t)$ and $\Div\DeltaTheta=-\bg$ on $\Gamma$.
Using Lemma~\ref{la_tr} and the fact that $z$, $\deltaz$, $\DeltaTheta\nn$ have zero trace on $\Gamma$
we deduce that
\[
   0 = \dualRM{\traceRM{\cT,t}(\deltaz,\DeltaTheta,\deltatau)}{(z,\TTheta,\btau)}{\cS,t}
     = \dualRM{\traceRM{\cT,t}(0,\Grad\bpsi,0)}{(z,\TTheta,\btau)}{\cS,t}
     = -t^2\dual{\bg}{\TTheta\nn}_\Gamma.
\]
Since $\bg\in\bH^{1/2}(\Gamma)$ was arbitrary we conclude that $\TTheta\in\HDivz{\Omega}$.
\end{enumerate}
\end{enumerate}
This finishes the proof.
\end{proof}

We continue to show that the minimum energy extension norm $\|\cdot\|_\trRM{\cS,t}$
cf.~\eqref{norm_tr_global}, is a product norm.

\begin{lemma} \label{la_split_norm}
Let $t>0$ and $a\in\{c,s\}$. The identity
\begin{align*}
   \|\tq\|_\trRM{\cS,t}^2
   =
   \sum_{T\in\cT} \|\tq\|_\trRM{\partial T,t}^2
   \quad\forall \tq\in \HRM{a}(\cS,t)
\end{align*}
holds true.
\end{lemma}

\begin{proof}
We use standard techniques, see, e.g., \cite{CarstensenDG_16_BSF,FuehrerHN_19_UFK}.

The inequality $\sum_{T\in\cT} \|\tq\|_\trRM{\partial T,t}^2 \le \|\tq\|_\trRM{\cS,t}^2$
is immediate by definition of the norms. Now, let $\tq=(\tq_T)_{T\in\cT}\in \HRM{a}(\cS,t)$ be given.
There exists $(z,\TTheta,\btau)\in \U{a}(t)$ such that $\traceRM{\cT,t}(z,\TTheta,\btau)=\tq$ and,
for $T\in\cT$, let $(\tilde z_T,\wilde\TTheta_T,\tilde\btau_T)\in V(T,t)$
be such that $\traceRM{T,t}(\tilde z_T,\wilde\TTheta_T,\tilde\btau_T)=\tq_T$ and
\[
   \|\tq_T\|_\trRM{\partial T,t}
   =
   \|(\tilde z_T,\wilde\TTheta_T,\tilde\btau_T)\|_{V(T,t)}.
\]
We find that $(\tilde z,\wilde\TTheta,\tilde\btau)\in V(\cT,t)$ defined by
$(\tilde z,\wilde\TTheta,\tilde\btau)|_T:=(\tilde z_T,\wilde\TTheta_T,\tilde\btau_T)$ ($T\in\cT$)
satisfies
\begin{align*}
\lefteqn{
   \dualRM{\traceRM{\cT,t}(\deltaz,\DeltaTheta,\deltatau)}{(\tilde z,\wilde\TTheta,\tilde\btau)}{\cS,t}
}\\
   &=
   \sum_{T\in\cT} \dualRM{\traceRM{T,t}(\deltaz,\DeltaTheta,\deltatau)}
                         {(\tilde z_T,\wilde\TTheta_T,\tilde\btau_T)}{\partial T,t}
   =
   -\sum_{T\in\cT} \dualRM{\traceRM{T,t}(\tilde z_T,\wilde\TTheta_T,\tilde\btau_T)}
                          {(\deltaz,\DeltaTheta,\deltatau)}{\partial T,t}
   \\
   &=
   -\dualRM{\tq}{(\deltaz,\DeltaTheta,\deltatau)}{\cS,t}
   =
   -\dualRM{\traceRM{\cT,t}(z,\TTheta,\btau)}{(\deltaz,\DeltaTheta,\deltatau)}{\cS,t}
   =0\quad\forall (\deltaz,\DeltaTheta,\deltatau)\in \U{a}(t)
\end{align*}
by Proposition~\ref{prop_jump}, so that $(\tilde z,\wilde\TTheta,\tilde\btau)\in \U{a}(t)$ also by
Proposition~\ref{prop_jump}. We conclude that
\begin{align*}
   \sum_{T\in\cT} \|\tq\|_\trRM{\partial T,t}^2
   &=
   \sum_{T\in\cT} \|(\tilde z_T,\wilde\TTheta_T,\tilde\btau_T\|_{V(T,t)}^2
   =
   \|(\tilde z,\wilde\TTheta,\tilde\btau)\|_{V(\cT,t)}^2
   \ge
   \|\tq\|_\trRM{\cS,t}^2
\end{align*}
where the last bound is due to the definition of the norm $\|\cdot\|_\trRM{\cS,t}$.
This finishes the proof.
\end{proof}

Finally we show that the norms $\|\cdot\|_{V(\cT,t)'}$ and
$\|\cdot\|_{\trRM{\cS,t}}$ are identical in $\HRM{a}(\cS,t)$ ($a\in\{c,s\}$). This is the
product variant of Lemma~\ref{la_trT_norms}.

\begin{prop} \label{prop_tr_norms}
Let $t>0$ and $a\in\{c,s\}$. It holds the identity
\[
   \|\tq\|_{V(\cT,t)'} = \|\tq\|_\trRM{\cS,t}
   \quad\forall\tq\in\HRM{a}(\cS,t).
\]
In particular,
\[
   \traceRM{\cT,t}:\; \U{a}(t)\to \HRM{a}(\cS,t)
\]
has unit norm and $\HRM{a}(\cS,t)$ is closed.
\end{prop}

\begin{proof}
With the preparations at hand the proof follows standard product arguments,
cf., e.g.,~\cite[Theorem~2.3]{CarstensenDG_16_BSF}, \cite[Proposition~3.5]{FuehrerHN_19_UFK}.
For convenience of the reader let us recall the arguments.

Let $\tq=(\tq_T)_{T\in\cT}\in \HRM{a}(\cS,t)$ be given.
Using Lemmas~\ref{la_split_norm} and~\ref{la_trT_norms} we calculate
\begin{align*}
   \|\tq\|_{V(\cT,t)'}^2
   &=
   \left(\sup_{0\not=(z,\TTheta,\btau)\in V(\cT,t)}
   \frac{\sum_{T\in\cT}\dualRM{\tq_T}{(z,\TTheta,\btau)}{\partial T,t}}
        {\|(z,\TTheta,\btau)\|_{V(\cT,t)}}
   \right)^2
   =
   \sum_{T\in\cT}
   \sup_{0\not=(z,\TTheta,\btau)\in V(T,t)}
   \frac{\dualRM{\tq_T}{(z,\TTheta,\btau)}{\partial T,t}^2}{\|(z,\TTheta,\btau)\|_{V(T,t)}^2}
   \\
   &=
   \sum_{T\in\cT} \|\tq_T\|_{V(T,t)'}^2
   =
   \sum_{T\in\cT} \|\tq_T\|_\trRM{\partial T,t}^2
   =
   \|\tq\|_\trRM{\cS,t}^2.
\end{align*}
Since $\HRM{a}(\cS,t)$ is the image of a bounded below operator, it is closed.
\end{proof}

\subsection{The Kirchhoff--Love case ($t=0$)} \label{sec_jumps_KLove}

In the following we collect the definitions and properties of spaces, norms and traces
from this section in the limit $t=0$, which is the Kirchhoff--Love case.
For the clamped plate, the corresponding results are taken from \cite{FuehrerHN_19_UFK},
whereas for the simply supported plate we have to introduce spaces that reflect this boundary condition.

Let us start collecting spaces and norms (the defined terms are those from \cite{FuehrerHN_19_UFK} in the
notation introduced there). For any $T\in\cT$ we have the space
\begin{align} \label{norm_V0}
   &H^2(T)\times\HdDiv{T}
   :=
   \{(z,\TTheta);\; (z,\TTheta,0)\in V(T,0)\}\quad \text{with norm}\nonumber\\
   &\|z\|_{2,T}^2 + \|\TTheta\|_{\dDiv,T}^2
   :=
   \|z\|_T^2 + \|\Grad\grad z\|_T^2 + \|\TTheta\|_T^2 + \|\div\Div\TTheta\|_T^2
   =
   \|(z,\TTheta,0)\|_{V(T,0)}^2.
\end{align}
That is, $V(T,0)=H^2(T)\times \HdDiv{T}$ is the quotient space with respect to the third component.
Correspondingly, there is the product space
\begin{align*}
   &H^2(\cT)\times\HdDiv{\cT}
   :=
   \{(z,\TTheta);\; (z,\TTheta,0)\in V(\cT,0)\}
\end{align*}
with squared norm $\|z\|_{2,\cT}^2 + \|\TTheta\|_{\dDiv,\cT}^2$, and the global quotient space $\U{}(0)$ 
with squared norm $\|z\|_2^2 + \|\TTheta\|_\dDiv^2$.
In the following, we simply drop the third component and refer to the quotient spaces as
\[
   V(\cT,0) = H^2(\cT)\times\HdDiv{\cT},\quad
   \U{}(0) = H^2(\Omega)\times\HdDiv{\Omega}.
\]
Note that, when $t=0$, the boundary conditions \eqref{Uc_BC} and \eqref{Us_BC}
become $z=0,\nn\cdot\grad z=0$ ($a=c$) and $z=0,\nn\cdot\TTheta\nn=0$ ($a=s$) on $\Gamma$, respectively.
Therefore, we define $\U{c}(0) := H^2_0(\Omega)\times\HdDiv{\Omega}$ needed in the case that
$a=c$ but, in order to define the space $\U{s}(\Omega)$ corresponding to $a=s$,
we have to give $\nn\cdot\TTheta\nn=0$ on $\Gamma$ a meaning when $\TTheta\in\HdDiv{\Omega}$.

\begin{remark} \label{rem_RM0_reg}
Lemma~\ref{la_reg} does not apply in the case $t=0$. Indeed, as shown in \cite{FuehrerHN_19_UFK}
by a counterexample, $(z,\TTheta)\in \U{}(0)$ does not imply $\TTheta\in\HDiv{\Omega}$.
Though, $(z,\TTheta,\btau)\in \U{c}(0)$ does mean that $z\in H^2_0(\Omega)$ and $\TTheta\in\HdDiv{\Omega}$,
and $(z,\TTheta,\btau)\in V(\cT,0)$ iff $z\in H^2(\cT)$ and $\TTheta\in\HdDiv{\cT}$.
\end{remark}

To consider the setting for $t=0$ we recall the following trace operators from \cite{FuehrerHN_19_UFK},
\begin{align} \label{trGG}
   \traceGG{}:\;
   \begin{cases}
      H^2(\Omega) &\to\ \HdDiv{\cT}',\\
      \quad z    &\mapsto\ \dual{\traceGG{}(z)}{\DeltaTheta}_{\cS}
                    := \vdual{z}{\div\Div\DeltaTheta}_\cT - \vdual{\Grad\grad z}{\DeltaTheta},
   \end{cases}
\end{align}
and
\begin{align} \label{trDD}
   \traceDD{}:\;
   \begin{cases}
      \HdDiv{\Omega} &\to\ H^2(\cT)',\\
      \qquad z    &\mapsto\ \dual{\traceDD{}(\TTheta)}{\deltaz}_{\cS}
                   := \vdual{\div\Div\TTheta}{\deltaz} - \vdual{\TTheta}{\Grad\grad\deltaz}_\cT.
   \end{cases}
\end{align}
For brevity we define $H^2_s(\Omega):=H^2(\Omega)\cap H^1_0(\Omega)$, and introduce the space
\begin{equation} \label{HdDivz_def}
   \HdDivz{\Omega} := \{\TTheta\in\HdDiv{\Omega};\; \dual{\traceDD{}(\TTheta)}{\deltaz}_\cS=0
                                                    \ \forall \deltaz\in H^2_s(\Omega)\}.
\end{equation}
Note that, for $\TTheta\in\HDiv{\Omega}\cap\HdDiv{\Omega}$,
$\dual{\traceDD{}(\TTheta)}{\deltaz}_\cS=-\dual{\TTheta\nn}{\grad\deltaz}_\Gamma$
for any $\deltaz\in H^2_s(\Omega)$ where the latter duality is the standard pairing
between $\bH^{-1/2}(\Gamma)$ and $\bH^{1/2}(\Gamma)$.
Then, taking into account that $\deltaz\in H^2_s(\Omega)$ has zero trace on $\Gamma$,
$\dual{\TTheta\nn}{\grad\deltaz}_\Gamma=0$ for any $\deltaz\in H^2_s(\Omega)$, this means that
$\nn\cdot\TTheta\nn=0$ on $\Gamma$ for a sufficiently smooth function $\TTheta$.
For a detailed discussion of the components of $\traceDD{}$ we refer to \cite{FuehrerHN_19_UFK}.

We are ready to define the trace spaces needed for the Kirchhoff--Love problems.
For the clamped plate we introduce
\[
   \bH^{3/2,1/2}_{00}(\cS) := \traceGG{}(H^2_0(\Omega)),\quad
   \bH^{-3/2,-1/2}(\cS) := \traceDD{}(\HdDiv{\Omega})
\]
whereas, for the simply supported plate, we need the spaces
\[
   \bH^{3/2,1/2}_{0}(\cS) := \traceGG{}(H^2_s(\Omega)),\quad
   \bH^{-3/2,-1/2}_{0}(\cS) := \traceDD{}(\HdDivz{\Omega}).
\]
These trace spaces are provided with canonical trace norms,
\begin{align*}
   &\|\tv\|_\trggrad{\cS}
   :=
   \inf\{\|v\|_2;\; v\in H^2(\Omega),\ \traceGG{}(v)=\tv\}
   && (\tv\in \bH^{3/2,1/2}_{0}(\cS)),\\
   &\|\tq\|_\trddiv{\cS}
   :=
   \inf \Bigl\{\|\TTheta\|_\dDiv;\; \TTheta\in \HdDiv{\Omega},\ \traceDD{}(\TTheta)=\tq\Bigr\}
   && (\tq\in \bH^{-3/2,-1/2}(\cS))
\end{align*}
(note that $\bH^{3/2,1/2}_{00}(\cS)\subset\bH^{3/2,1/2}_{0}(\cS)$ and
$\bH^{-3/2,-1/2}_{0}(\cS)\subset\bH^{-3/2,-1/2}(\cS)$ are closed subspaces furnished with the same
respective norm).
Now, setting $t=0$, the Reissner--Mindlin trace operator reveals two components,
\begin{align} \label{trRM0}
\lefteqn{
   \dualRM{\traceRM{\cT,0}(z,\TTheta,\btau)}{(\deltaz,\DeltaTheta,\deltatau)}{\cS,0}
}\nonumber\\
   &=
   \vdual{z}{\div\Div\DeltaTheta}_\cT
   -\vdual{\div\Div\TTheta}{\deltaz}
   +\vdual{\TTheta}{\Grad\grad\deltaz}_\cT
   -\vdual{\Grad\grad z}{\DeltaTheta}
   \nonumber\\
   &=
   \dual{\traceGG{}(z)}{\DeltaTheta}_\cS - \dual{\traceDD{}(\TTheta)}{\deltaz}_\cS
\end{align}
for $(z,\TTheta,\btau)\in \U{}(0)$ and $(\deltaz,\DeltaTheta,\deltatau)\in V(\cT,0)$.
In the following we again drop the third argument and write $\traceRM{\cT,0}(z,\TTheta)$ instead of
$\traceRM{\cT,0}(z,\TTheta,\btau)$. Thus, we have a trace operator with two independent components,
\[
   \traceRM{\cT,0}:\;
   \left\{\begin{array}{cll}
      \U{}(0) & \to & V(\cT,0)',\\[1em]
      (z,\TTheta) & \mapsto & \dual{\traceRM{\cT,0}(z,\TTheta)}{(\deltaz,\DeltaTheta)}_\cS
                              =
                              \dual{\traceGG{}(z)}{\DeltaTheta}_\cS
                              -
                              \dual{\traceDD{}(\TTheta)}{\deltaz}_\cS
   \end{array}\right..
\]
Then, defining
\begin{equation} \label{spaces_t0}
   \U{c}(0) := H^2_0(\Omega)\times\HdDiv{\Omega} \quad\text{and}\quad
   \U{s}(0) := H^2_s(\Omega)\times\HdDivz{\Omega}
\end{equation}
($\U{c}(0)$ had been defined previously) our trace spaces are
\begin{align} \label{trace_space_RMc0}
   \HRM{c}(\cS,0) := \traceRM{\cT,0}(\U{c}(t))
   &= \bH^{3/2,1/2}_{00}(\cS) \times \bH^{-3/2,-1/2}(\cS)
   \quad\text{(clamped)}
\end{align}
and
\begin{align} \label{trace_space_RMs0}
   \HRM{s}(\cS,0) := \traceRM{\cT,0}(\U{s}(t))
   &=: \bH^{3/2,1/2}_{0}(\cS) \times \bH^{-3/2,-1/2}_{0}(\cS)
   \quad\text{(simple support)}
\end{align}
with the canonical trace norm
\begin{align} \label{norm_tr}
   \|(\tv,\tq)\|_\trRM{\cS,0}
   &= \Bigl(\|\tv\|_\trggrad{\cS}^2 + \|\tq\|_\trddiv{\cS}^2\Bigr)^{1/2}.
\end{align}
The dualities between $\HRM{a}(\cS,0)$ ($a\in\{c,s\}$) and $V(\cT,0)$ are given by the respective component
dualities,
\begin{align} \label{duality_RM0}
    \dualRM{(\tv,\tq)}{(\deltaz,\DeltaTheta)}{\cS,0}
    =
    \dual{\tv}{\DeltaTheta}_\cS - \dual{\tq}{\deltaz}_\cS
    :=
    \dual{\traceGG{}(z)}{\DeltaTheta}_\cS - \dual{\traceDD{}(\TTheta)}{\deltaz}_\cS
\end{align}
for $(\tv,\tq)\in\HRM{a}(\cS,0)$, $(\deltaz,\DeltaTheta)\in V(\cT,0)$ and any
$(z,\TTheta)\in \U{a}(0)$ with $\traceRM{\cT,0}(z,\TTheta)=(\tv,\tq)$, cf.~\eqref{trRM0}.
They give rise to the duality norms
\begin{align*}
   \|(\tv,\tq)\|_{V(\cT,0)'}
   &=
   \sup_{0\not=(\deltaz,\DeltaTheta)\in V(\cT,0)}
   \frac{\dualRM{(\tv,\tq)}{(\deltaz,\DeltaTheta)}{\cS,0}}{\|(\deltaz,\DeltaTheta)\|_{V(\cT,0)}}
   \quad \forall(\tv,\tq)\in \HRM{a}(\cS,0)\quad(a\in\{c,s\}).
\end{align*}
In the following we collect some technical results.

\begin{lemma} \label{la_trRM0_asym}
The trace operator $\traceRM{\cT,0}$ satisfies the relation
\begin{align*} 
   \dualRM{\traceRM{\cT,0}(z,\TTheta)}{(\deltaz,\DeltaTheta)}{\cS,0}
   =
   -\dualRM{\traceRM{\cT,0}(\deltaz,\DeltaTheta)}{(z,\TTheta)}{\cS,0}
\end{align*}
for any $(z,\TTheta), (\deltaz,\DeltaTheta)\in U(0)$.
\end{lemma}

\begin{proof}
The stated relation follows from \eqref{trRM0} by noting that
\begin{align} \label{trGGDD}
   \dual{\traceGG{}(z)}{\TTheta}_\cS=\dual{\traceDD{}(\TTheta)}{z}_\cS
   \quad\forall z\in H^2(\Omega), \TTheta\in\HdDiv{\Omega}
\end{align}
by definitions \eqref{trGG},\eqref{trDD}, see also \cite[(3.14)]{FuehrerHN_19_UFK}.
\end{proof}

\begin{remark}
Whereas the relation of Lemma~\ref{la_trRM0_asym} corresponds to relation \eqref{trT_asym}
of Lemma~\ref{la_trT}, the decomposition \eqref{trT_rep} and Lemma~\ref{la_tr} do not apply
in the case $t=0$. This is due to the lacking regularity of $\TTheta$ for $(z,\TTheta)\in V(T,0)$
or $(z,\TTheta)\in\U{}(0)$, cf.~Remark~\ref{rem_RM0_reg}.
\end{remark}

The following result is Proposition~\ref{prop_jump} in the case $t=0$.

\begin{prop} \label{prop_jump0}
For $a\in\{c,s\}$ and $(z,\TTheta)\in V(\cT,0)$ it holds
\[
   (z,\TTheta)\in \U{a}(0) \quad\Leftrightarrow\quad
   \dualRM{\traceRM{\cT,0}(\deltaz,\DeltaTheta)}{(z,\TTheta)}{\cS,0} = 0
   \quad\forall (\deltaz,\DeltaTheta)\in \U{a}(0).
\]
\end{prop}

\begin{proof}
We consider the case $a=c$. Let $(z,\TTheta)\in V(\cT,0)$ be given.
By \cite[Proposition~3.4(i)]{FuehrerHN_19_UFK}, $\TTheta\in\HdDiv{\Omega}$ if and only if
$\dual{\traceGG{}(\deltaz)}{\TTheta}_\cS=0$ for any $\deltaz\in H^2_0(\Omega)$.
Also, by \cite[Proposition~3.8(i)]{FuehrerHN_19_UFK},
$z\in H^2_0(\Omega)$ iff $\dual{\traceDD{}(\DeltaTheta)}{z}_\cS=0$ for any
$\DeltaTheta\in\HdDiv{\Omega}$. Since $\U{c}(0)=H^2_0(\Omega)\times\HdDiv{\Omega}$,
\eqref{trRM0} gives the statement (interchanging $(z,\TTheta)$ and $(\deltaz,\DeltaTheta)$ there).

Now we consider the case $a=s$. For $(z,\TTheta), (\deltaz,\DeltaTheta)\in \U{s}(0)$ we obtain
\begin{align*}
   \dualRM{\traceRM{\cT,0}(\deltaz,\DeltaTheta)}{(z,\TTheta)}{\cS,0}
   &\overset{\mathrm{def}}=
   \vdual{\deltaz}{\div\Div\TTheta}_\cT
   -\vdual{\div\Div\DeltaTheta}{z}
   +\vdual{\DeltaTheta}{\Grad\grad z}_\cT
   -\vdual{\Grad\grad\deltaz}{\TTheta}
   \\
   &=
   \vdual{\deltaz}{\div\Div\TTheta}
   -\vdual{\div\Div\DeltaTheta}{z}
   +\vdual{\DeltaTheta}{\Grad\grad z}
   -\vdual{\Grad\grad\deltaz}{\TTheta}
   \\
   &=
   \dual{\traceDD{}(\TTheta)}{\delta z}_\cS
   -
   \dual{\traceDD{}(\DeltaTheta)}{z}_\cS = 0,
\end{align*}
cf.~\eqref{HdDivz_def}. On the other hand, if
\[
   \dualRM{\traceRM{\cT,0}(\deltaz,\DeltaTheta)}{(z,\TTheta)}{\cS,0} = 0
   \quad\forall (\deltaz,\DeltaTheta)\in \U{s}(0)\cap\U{c}(0),
\]
one deduces that $(z,\TTheta)\in U(0)$ as in the case $a=c$. Then,
\[
   \dualRM{\traceRM{\cT,0}(\deltaz,0)}{(z,\TTheta)}{\cS,0}
   =
   \dual{\traceDD{}(\TTheta)}{\delta z}_\Gamma
   =
   0
   \quad\forall \deltaz\in H^2_s(\Omega)
\]
reveals that $\TTheta\in\HdDivz{\Omega}$ by definition \eqref{HdDivz_def}, and
\[
   \dualRM{\traceRM{\cT,0}(0,\DeltaTheta)}{(z,\TTheta)}{\cS,0}
   =
   -\dual{\traceDD{}(\DeltaTheta)}{z}_\cS
   =
   -\dual{\traceGG{}(z)}{\DeltaTheta}_\Gamma
   =
   0
   \quad\forall \DeltaTheta\in\HdDivz{\Omega}
\]
shows that $z=0$ on $\Gamma$ by density since
$\dual{\traceGG{}(z)}{\DeltaTheta}_\Gamma=\dual{\nn\cdot\Div\DeltaTheta}{z}_\Gamma$
for smooth tensors $\DeltaTheta\in\HdDivz{\Omega}$. For details we refer to the
proof of \cite[Proposition~3.8(i)]{FuehrerHN_19_UFK}.

Together, we have shown that $(z,\TTheta)\in \U{s}(0)$.
\end{proof}

\begin{cor} \label{cor_jump0}
Let $z\in H^2(\Omega)$. Then, $z\in H^2_s(\Omega)$ if and only if
\[
   \dual{\traceGG{}(z)}{\DeltaTheta}_\Gamma = 0\quad\forall\DeltaTheta\in\HdDivz{\Omega}.
\]
\end{cor}

\begin{proof}
This is the statement of Proposition~\ref{prop_jump0} when selecting $a=s$ and test functions
$(\deltaz,\DeltaTheta)=(0,\DeltaTheta)\in \U{s}(0)=H^2_s(\Omega)\times\HdDivz{\Omega}$.
One only has to note the splitting \eqref{trRM0} of the trace operator $\traceRM{\cT,0}$.
\end{proof}

Next we formulate Proposition~\ref{prop_tr_norms} in the case $t=0$.

\begin{prop} \label{prop_trRM0_norms}
It holds the identity
\[
   \|(\tv,\tq)\|_\trRM{\cS,0} = \|(\tv,\tq)\|_{V(\cT,0)'}
   \quad\forall (\tv,\tq)\in\HRM{a}(\cS,0),\ a\in\{c,s\}.
\]
In particular,
\[
   \traceRM{\cT,0}:\; \U{a}(0)\to \HRM{a}(\cS,0)\quad (a\in\{c,s\})
\]
have unit norm and $\HRM{a}(\cS,0)$ ($a\in\{c,s\}$) are closed.
\end{prop}

\begin{proof}
Using the product property $V(\cT,0)=H^2(\cT)\times\HdDiv{\cT}$ and duality \eqref{duality_RM0}, the relation
\begin{align*}
   \|(\tv,\tq)\|_{V(\cT,0)'}^2
   &=
   \sup_{0\not=\DeltaTheta\in\HdDiv{\cT}}
   \frac{\dual{\tv}{\DeltaTheta}_\cS^2}{\|\DeltaTheta\|_{\dDiv,\cT}^2}
   +
   \sup_{0\not=\deltaz\in H^2(\cT)}
   \frac{\dual{\tq}{\deltaz}_\cS^2}{\|\deltaz\|_{2,\cT}^2}
   \\
   &=:
   \|\tv\|_{(\dDiv,\cT)'}^2 + \|\tq\|_{(2,\cT)'}^2
   \quad \forall(\tv,\tq)\in \HRM{c}(\cS,0)\cup \HRM{s}(\cS,0)
\end{align*}
holds. Then, the statement is a combination of relation \eqref{norm_tr}
with the identities
\begin{align}
   \label{pfc1}
   \|\tq\|_{(2,\cT)'} &= \|\tq\|_\trddiv{\cS}\quad (\tq\in\bH^{-3/2,-1/2}(\cS)),\\
   \label{pfc2}
   \|\tv\|_{(\dDiv,\cT)'} &= \|\tv\|_\trggrad{\cS}\quad (\tv\in\bH^{3/2,1/2}_{00}(\cS))
\end{align}
when $a=c$, and
\begin{align}
   \label{pfs1}
   \|\tq\|_{(2,\cT)'} &= \|\tq\|_\trddiv{\cS}\quad (\tq\in\bH^{-3/2,-1/2}_{0}(\cS)),\\
   \label{pfs2}
   \|\tv\|_{(\dDiv,\cT)'} &= \|\tv\|_\trggrad{\cS}\quad (\tv\in\bH^{3/2,1/2}_{0}(\cS))
\end{align}
when $a=s$. The former identities are true by Propositions~3.5 and~3.9, respectively,
from \cite{FuehrerHN_19_UFK}. Furthermore, \eqref{pfc1} implies \eqref{pfs1} since
$\bH^{-3/2,-1/2}_{0}(\cS)\subset \bH^{-3/2,-1/2}(\cS)$, and inspection reveals
that the proof of \eqref{pfc2} by \cite[Proposition~3.9]{FuehrerHN_19_UFK} also applies to
\eqref{pfs2}. (We remark that in \cite{FuehrerHN_19_UFK} our norm $\|\cdot\|_\trggrad{\cS}$
is referred to as $\|\cdot\|_\trggrad{0,\cS}$ in $\bH^{3/2,1/2}_{00}(\cS)$.)
\end{proof}

\section{Variational formulation and DPG method} \label{sec_uw}

Having all the necessary spaces and norm relations at hand we return to the construction
of an ultraweak formulation of the Reissner--Mindlin problem. Recall the preliminary formulation \eqref{VFa}.
From Lemma~\ref{la_trT} it is now clear that the interface terms in \eqref{VFa} can be represented as
\begin{align*}
   &\sum_{T\in\cT} \dual{u}{\nn\cdot(\Div\TTheta+t(\btau-\grad z))}_{\partial T}
    -\sum_{T\in\cT} \dual{\nn\cdot(\Div\MM+t(\btheta-\grad u))}{z}_{\partial T}
   \\
   &+\sum_{T\in\cT} \dual{\MM\nn}{\grad z-t^2\Div\TTheta}_{\partial T}
    -\sum_{T\in\cT} \dual{\grad u-t^2\Div\MM}{\TTheta\nn}_{\partial T}
   \\
   =&\
   \sum_{T\in\cT} \dualRM{\traceRM{T,t}(u,\MM,\btheta)}{(z,\TTheta,\btau)}{\partial T,t}
   =
   \dualRM{\traceRM{\cT,t}(u,\MM,\btheta)}{(z,\TTheta,\btau)}{\cS,t}.
\end{align*}
We introduce the independent trace variable $\tq:=\traceRM{\cT,t}(u,\MM,\btheta)$ and define the spaces
\begin{align*}
   &\U{a}(\cT,t) := L_2(\Omega)\times\LL_2^s(\Omega)\times\bL_2(\Omega)\times \HRM{a}(\cS,t)
   \quad (t>0, a\in\{c,s\}),\\
   &\U{a}(\cT,0) := L_2(\Omega)\times\LL_2^s(\Omega)\times\{0\}\times \HRM{a}(\cS,0)
   \quad (a\in\{c,s\}).
\end{align*}
Here, $\U{a}(\cT,0)$ is understood as being the corresponding quotient space with respect
to the third component $\bL_2(\Omega)$, and we recall \eqref{trace_space_RMc0}, \eqref{trace_space_RMs0}
for the definition of $\HRM{a}(\cS,0)$.
We consider the norm
\begin{align} \label{norm_U}
   \|(u,\MM,\btheta,\tq)\|_{U(\cT,t)}
   &:=
   \Bigl(\|u\|^2 + \|\MM\|^2 + t\|\btheta\|^2 + \|\tq\|_{V(\cT,t)'}^2\Bigr)^{1/2}
   \quad (t\ge 0).
\end{align}
With the preparations in Section~\ref{sec_jumps_KLove}, we are able to consider the thickness parameter
$t$ including the case $t=0$, which represents the Kirchhoff--Love model.

Our ultraweak variational formulation of \eqref{prob} with boundary condition \eqref{pBCc} ($a=c$, clamped)
or \eqref{pBCs} ($a=s$, simple support) is:
\emph{For given $f\in L_2(\Omega)$ and $t\in [0,1]$, find $(u,\MM,\btheta,\tq)\in \U{a}(\cT,t)$ such that}
\begin{align} \label{VF}
   b_t(u,\MM,\btheta,\tq;z,\TTheta,\btau) = L(z,\TTheta,\btau)
   \quad\forall (z,\TTheta,\btau)\in V(\cT,t).
\end{align}
Here,
\begin{align} \label{b}
   b_t(u,\MM,\btheta,\tq;z,\TTheta,\btau)
   :=
   &\vdual{u}{\div(\Div\TTheta+t(\btau-\grad z))}_\cT
   +\vdual{\MM}{\cCinv\TTheta+\Grad(\grad z-t^2\Div\TTheta)}_\cT
   \nonumber\\
   &+t\vdual{\btheta}{\btau-\grad z}_\cT
   - \dualRM{\tq}{(z,\TTheta,\btau)}{\cS,t}
\end{align}
and
\begin{align*}
   L(z,\TTheta,\btau) := -\vdual{f}{z}.
\end{align*}
In the case $t=0$, the bilinear form reduces to
\[
   b_0(u,\MM,\btheta,\tq;z,\TTheta,\btau)
   =
   \vdual{u}{\div\Div\TTheta}_\cT
   +\vdual{\MM}{\cCinv\TTheta+\Grad\grad z}_\cT
   - \dualRM{\tq}{(z,\TTheta)}{\cS,0}
\]
with $\tq\in\HRM{a}(\cS,0)$, a quotient space.
Recall that the skeleton duality in \eqref{b} is defined by \eqref{tr_RM_dual} ($t>0$)
and \eqref{duality_RM0} ($t=0$). Furthermore, we recall the definition of $V(\cT,t)$
as a product space through completion with respect to the component norm \eqref{norm_VT}.
This applies to $t\in [0,1]$, see also \eqref{norm_V0} for the case $t=0$.

One of our main results is the following theorem.

\begin{theorem} \label{thm_stab}
Let $a\in\{c,s\}$. For any function $f\in L_2(\Omega)$ and any $t\in [0,1]$, there exists
a unique solution $(u,\MM,\btheta,\tq)\in \U{a}(\cT,t)$ to \eqref{VF}. It is uniformly bounded,
\[
   \|(u,\MM,\btheta,\tq)\|_{U(\cT,t)} \lesssim \|f\|
\]
with a hidden constant that is independent of $f$, $\cT$, and $t\in [0,1]$.
Furthermore, $(u,\MM,\btheta)\in \U{a}(t)$ solves \eqref{prob} and
$\tq=\traceRM{\cT,t}(u,\MM,\btheta)$.
\end{theorem}

A proof of this theorem is given in Section~\ref{sec_pf}.
A consequence of Theorem~\ref{thm_stab} is that the solution of \eqref{VF} converges weakly to
the solution of the corresponding Kirchhoff--Love formulation when $t\to 0$.

\begin{theorem} \label{thm_limit}
Let $a\in\{c,s\}$, and let $f_t\in L_2(\Omega)$ for $t\in [0,1]$ with $f_t\to f_0$ in $L_2(\Omega)$ when $t\to 0$.
Furthermore, for $t\in [0,1]$, let $(u_t,\MM_t,\btheta_t,\tq_t)\in \U{a}(\cT,t)$ be the solution of \eqref{VF}.
It holds
\[
   (u_t,\MM_t) \rightharpoonup (u_0,\MM_0)
   \quad\text{in}\quad L_2(\Omega)\times \LL_2^s(\Omega)
   \quad (t\to 0),
\]
\[
   \div\Div\MM_t \to \div\Div\MM_0
   \quad\text{in}\quad L_2(\Omega),
\]
and
\[
   \dualRM{\tq_t}{(z,\TTheta,\btau)}{\cS,t} \to \dualRM{\tq_0}{(z,\TTheta)}{\cS,0} \quad (t\to 0)
\]
for any $(z,\TTheta,\btau)\in V(\cT,1)\cap V(\cT,0)$.
\end{theorem}

We prove this statement in Section~\ref{sec_pf_limit}.

Now, to invoke the DPG method, we consider a (family of) discrete subspace(s)
$\U{a,h}(\cT,t)\subset \U{a}(\cT,t)$ ($a=c$ or $a=s$, depending on the boundary condition),
and define the \emph{trial-to-test operator} $\ttt_t:\;\U{a}(\cT,t)\to V(\cT,t)$ by
\begin{equation} \label{ttt}
   \ip{\ttt_t(\uu)}{\bv}_{V(\cT,t)} = b_t(\uu,\bv)\quad\forall\bv\in V(\cT,t).
\end{equation}
Then, the DPG method with optimal test functions for problem~\eqref{prob} (and based
on the variational formulation~\eqref{VF}) is:
\emph{Find $\uu_h\in \U{a,h}(\cT,t)$ such that}
\begin{align} \label{DPG}
   b_t(\uu_h,\ttt_t\bdeltau) = L(\ttt_t\bdeltau) \quad\forall\bdeltau\in \U{a,h}(\cT,t).
\end{align}
This discretization scheme is a minimum residual method. Defining the operator
$B_t:\;\U{a}(\cT,t)\to V(\cT,t)'$ by $B_t(\uu)(\bv):=b_t(\uu,\bv)$,
the DPG scheme delivers the best approximation with respect to the
so-called \emph{energy norm} $\|\cdot\|_{E(\cT,t)}:=\|B_t(\cdot)\|_{V(\cT,t)'}$,
cf., e.g.,~\cite{DemkowiczG_11_ADM}.

Our second main result is the uniform quasi-optimal convergence of the DPG scheme \eqref{DPG}
in the $U(\cT,t)$-norm.

\begin{theorem} \label{thm_DPG}
Let $a\in\{c,s\}$, $f\in L_2(\Omega)$ and $t\in [0,1]$ be given.
For any finite-dimensional subspace $\U{a,h}(\cT,t)\subset \U{a}(\cT,t)$
there exists a unique solution $\uu_h\in \U{a,h}(\cT,t)$ to \eqref{DPG}. It satisfies the quasi-optimal
error estimate
\[
   \|\uu-\uu_h\|_{U(\cT,t)} \lesssim \|\uu-\bw\|_{U(\cT,t)}
   \quad\forall\bw\in \U{a,h}(\cT,t)
\]
with a hidden constant that is independent of $f$, $\cT$, $t\in [0,1]$, and $\U{a,h}(\cT,t)$.
\end{theorem}

A proof of this theorem is given in Section~\ref{sec_pf}.

\section{Inf-sup conditions and proofs of Theorems~\ref{thm_stab},~\ref{thm_limit},~\ref{thm_DPG}}
\label{sec_adj}

The proof of Theorem~\ref{thm_stab} follows standard techniques for mixed formulations.
In the context of product (or ``broken'') test spaces, the literature offers three variants.
Initially the whole adjoint problem was analyzed by subdividing it into one without jumps
and a homogeneous one with jump data. Showing stability of the latter one requires to
construct a Helmholtz decomposition, cf., e.g.,~\cite{DemkowiczG_11_ADM}.
Another technique is to analyze the adjoint problem as a whole in the form of
a mixed problem but without Lagrangian multiplier, cf.~\cite{FuehrerHN_19_UFK}.
Here, we follow the strategy from Carstensen \emph{et al.} \cite{CarstensenDG_16_BSF}
where the first approach of splitting the adjoint problem has been analyzed in an abstract
way, thus avoiding the construction of a Helmholtz decomposition.
Still, one of the main ingredients is to prove stability of the adjoint problem without jumps.
In our case, taking the $t$-weighting in the $U(\cT,t)$-norm into account (cf.~\eqref{norm_U}),
it reads as follows.
\emph{Find $(z,\TTheta,\btau)\in \U{a}(t)$ such that}
\begin{subequations} \label{adj}
\begin{alignat}{2}
   \div(\Div\TTheta+t(\btau-\grad z))     &= g    && \ \in L_2(\Omega)\label{ar1},\\
    \cCinv\TTheta+\Grad(\grad z-t^2\Div\TTheta) &= \HH  && \ \in \LL_2^s(\Omega)\label{ar2},\\
    t^{1/2}(\btau-\grad z)                      &= \bxi && \ \in \bL_2(\Omega)\label{ar3}.
\end{alignat}
\end{subequations}
We show that this problem is well posed.

\begin{lemma} \label{la_adj_well}
Let $a\in\{c,s\}$.
Assuming the compatibility $\bxi=0$ if $t=0$,
problem \eqref{adj} is uniformly well posed for $t\in [0,1]$. Its solution is bounded like
\[
   \|(z,\TTheta,\btau)\|_{U(t)} \lesssim \|g\| + \|\HH\| + \|\bxi\|
\]
with a constant that is independent of $t\in [0,1]$. In the case $t=0$, this means that
its solution is unique in the quotient space $\U{a}(0)$ with bound
\[
   \|(z,\TTheta)\|_{U(0)} \lesssim \|g\| + \|\HH\|.
\]
\end{lemma}

\begin{proof}
Let $t\in (0,1]$.
Recall that $(z,\TTheta,\btau)\in \U{a}(t)$ implies that
$\TTheta\in\HDiv{\Omega}$, $\grad z-t^2\Div\TTheta\in\bH^1_0(\Omega)$ if $a=c$, and
$\TTheta\in\HDivz{\Omega}$, $\grad z-t^2\Div\TTheta\in\bH^1(\Omega)$ if $a=s$, cf.~Lemma~\ref{la_reg}.

Therefore, testing \eqref{ar1} by $\deltaz\in H^1_0(\Omega)$ ($a\in\{c,s\}$),
\eqref{ar2} by $\DeltaTheta\in\HDiv{\Omega}$ ($a=c$) or $\DeltaTheta\in\HDivz{\Omega}$ ($a=s$),
and replacing $t(\btau-\grad z)=t^{1/2}\bxi$,
integration by parts yields the following variational formulation of \eqref{adj}.
\emph{Find $(z,\TTheta)\in H^1_0(\Omega)\times\HDiv{\Omega}$ such that
\begin{align} \label{adj_weak}
   \vdual{\cCinv\TTheta}{\DeltaTheta} + t^2 \vdual{\Div\TTheta}{\Div\DeltaTheta}
   -&\vdual{\grad z}{\Div\DeltaTheta} - \vdual{\Div\TTheta}{\grad\deltaz}
   \nonumber\\
   &=
   \vdual{g}{\deltaz} + \vdual{\HH}{\DeltaTheta} + t^{1/2}\vdual{\bxi}{\grad\deltaz}
\end{align}
for any $(\deltaz,\DeltaTheta)\in H^1_0(\Omega)\times\HDiv{\Omega}$}
if $a=c$, and the same except for replacing $\HDiv{\Omega}$ by $\HDivz{\Omega}$ if $a=s$.

Let us show that this formulation is well posed and that its solution gives rise to the solution of
\eqref{adj}. Since $\cC$ induces a self-adjoint isomorphism
$\LL_2^s(\Omega)\to\LL_2^s(\Omega)$, the term
$\vdual{\cCinv\TTheta}{\DeltaTheta} + t^2 \vdual{\Div\TTheta}{\Div\DeltaTheta}$ gives rise
to a uniformly bounded and coercive bilinear form in $\HDiv{\Omega}\times\HDiv{\Omega}$
with norm $\bigl(\|\cdot\|^2+t^2\|\Div\cdot\|^2\bigr)^{1/2}$. Now,
\[
   \Div:\;\HDiv{\Omega} \to \bL_2(\Omega)\quad\text{and}\quad
   \div:\;\bL_2(\Omega)\to (H^1_0(\Omega))'
\]
are surjective operators and so is their composition.
The surjectivity of $\dDiv:\;\HDiv{\Omega}\to (H^1_0(\Omega))'$
is equivalent to an inf-sup condition
\[
   \sup_{0\not=\TTheta\in\HDiv{\Omega}}
   \frac{\vdual{\Div\TTheta}{\grad\deltaz}}
	{\bigl(\|\TTheta\|^2 + \|\Div\TTheta\|^2\bigr)^{1/2}}
   \gtrsim
   \|\grad\deltaz\| \quad \forall\deltaz\in H^1_0(\Omega).
\]
Therefore, also the weaker estimate
\[
   \sup_{0\not=\TTheta\in\HDiv{\Omega}}
   \frac{\vdual{\Div\TTheta}{\grad\deltaz}}
	{\bigl(\|\TTheta\|^2 + t^2\|\Div\TTheta\|^2\bigr)^{1/2}}
   \gtrsim
   \|\grad\deltaz\| \quad \forall\deltaz\in H^1_0(\Omega)
\]
holds, with an implicit constant that is independent of $t\in (0,1]$.
Using the theory of mixed formulations we conclude that problem \eqref{adj_weak} is well posed
and that its solution is bounded like
\begin{align} \label{bound_tDiv}
   \|\TTheta\|^2 + t^2\|\Div\TTheta\|^2 + \|z\|^2 + t\|\grad z\|^2 \lesssim \|g\|^2 + \|\HH\|^2 + \|\bxi\|^2,
\end{align}
uniformly for $t\in (0,1]$.

Now, defining $\btau:=\grad z+t^{-1/2}\bxi$, $(z,\TTheta,\btau)$ is the unique solution of \eqref{adj}.
Indeed, \eqref{ar3} is satisfied by selection of $\btau$, and \eqref{ar1} holds as can be seen by
choosing $\DeltaTheta=0$ in \eqref{adj_weak} and replacing $\bxi=t^{1/2}(\btau-\grad z)$.
Finally, setting $\deltaz=0$ in \eqref{adj_weak} shows that \eqref{ar2} holds and, in particular,
$\grad z-t^2\Div\TTheta\in\bH^1_0(\Omega)$ if $a=c$ and
$\grad z-t^2\Div\TTheta\in\bH^1(\Omega)$ if $a=s$. It follows that $(z,\TTheta,\btau)\in \U{a}(t)$.

We bound the remaining terms,
\begin{alignat*}{2}
   &t^{1/2}\|\btau\| \le t^{1/2}\|\grad z\| + \|\bxi\|  \qquad&& \text{(by \eqref{ar3})},\\
   &\|\Grad(\grad z-t^2\Div\TTheta)\| \lesssim \|\HH\| + \|\TTheta\| \qquad&& \text{(by \eqref{ar2})},\\
   &\|\div(\Div\TTheta+t(\btau-\grad z))\| = \|\bg\|            \qquad&& \text{(by \eqref{ar1})}.
\end{alignat*}
This proves the statement for positive $t$. In the case that $t=0$ and $\bxi=0$, \eqref{adj} reads
\[
   \div\Div\TTheta = g,\qquad \cCinv\TTheta+\Grad\grad z = \HH.
\]
Eliminating $\TTheta$, it becomes $\div\Div\cC\Grad\grad z = \div\Div\cC\HH-g$, in weak form
\begin{equation} \label{adj0_prob_z}
   z\in H^2_m(\Omega):\quad
   \vdual{\cC\Grad\grad z}{\Grad\grad\deltaz} = \vdual{\cC\HH}{\Grad\grad\deltaz} - \vdual{g}{\deltaz}
   \quad\forall\deltaz\in H^2_m(\Omega)
\end{equation}
with $m=0$ if $a=c$ and $m=s$ if $a=s$ (recall that $H^2_s(\Omega)=H^2(\Omega)\cap H^1_0(\Omega)$).
Note that in the case $a=s$, this formulation includes the natural boundary condition
$\nn\cdot\cC(\Grad\grad z-\HH)\nn=0$ on $\Gamma$, that is,
$\cC(\Grad\grad z-\HH)\in\HdDivz{\Omega}$.

Problem \eqref{adj0_prob_z} has a unique solution since the bilinear form is coercive both
on $H^2_0(\Omega)$ and $H^2_s(\Omega)$. It holds the bound $\|z\|_2\lesssim \|\HH\| + \|g\|$.
The formulation also shows that $\TTheta=\cC(\HH-\Grad\grad z)\in\HdDiv{\Omega}$
($\HdDivz{\Omega}$ if $a=s$) with $\div\Div\TTheta=g$.
Recalling relation \eqref{spaces_t0} for $\U{a}(0)$, we therefore obtain a unique solution
$(z,\TTheta)\in\U{a}(0)$ of \eqref{adj} for $t=0$ with
\[
   \|(z,\TTheta)\|_{U(0)} = \bigl(\|z\|_2^2 + \|\TTheta\|_\dDiv^2\bigr)^{1/2} \lesssim \|g\| + \|\HH\|.
\]
This finishes the proof.
\end{proof}

\begin{cor} \label{cor_adj_well}
Let $t>0$. The bound
\[
   t\|\Div\TTheta\|
   \lesssim \|(z,\TTheta,\btau)\|_{U(t)} \quad\forall (z,\TTheta,\btau)\in \U{a}(t)\quad (a\in\{c,s\})
\]
holds with a constant that is independent of $t\in (0,1]$.
\end{cor}

\begin{proof}
This has been shown in the proof of Lemma~\ref{la_adj_well}.
We just need to apply the triangle inequality on the right-hand side of~\eqref{bound_tDiv}, giving
\begin{align*}
   t^2\|\Div\TTheta\|^2
   &\lesssim
   \|\div(\Div\TTheta+t(\btau-\grad z)\|^2 + \|\Grad(\grad z-t^2\Div\TTheta)\|^2 + \|\TTheta\|^2
   + t\|\btau\|^2 + t\|\grad z\|^2
   \\
   &\le
   \|(z,\TTheta,\btau)\|_{U(t)}^2.
\end{align*}
\end{proof}

Another ingredient to show well-posedness of \eqref{VF} is the injectivity of the adjoint operator
$B_t^*$. This is shown next.

\begin{lemma} \label{la_inj}
Let $a\in\{c,s\}$. For $t\in [0,1]$, the adjoint operator $B_t^*:\;V(\cT,t)\to \U{a}(\cT,t)'$ is injective.
\end{lemma}

\begin{proof}
Let $(z,\TTheta,\btau)\in V(\cT,t)$ be such that
$b_t(\deltaz,\DeltaTheta,\deltatau,\deltaq;z,\TTheta,\btau)=0$ for any
$(\deltaz,\DeltaTheta,\deltatau,\deltaq)\in \U{a}(\cT,t)$.
Selecting $\deltaz=0$, $\DeltaTheta=0$, $\deltatau=0$ and $\deltaq\in\HRM{a}(\cS,t)$,
Proposition~\ref{prop_jump} (if $t>0$) and Proposition~\ref{prop_jump0} (if $t=0$) show that
$(z,\TTheta,\btau)\in \U{a}(t)$. It follows that $(z,\TTheta,\btau)$ solves
\begin{alignat*}{2}
    \div(\Div\TTheta+t(\btau-\grad z))    &= 0,\quad
    \cCinv\TTheta+\Grad(\grad z-t^2\Div\TTheta) = 0,\quad
     t^{1/2}(\btau-\grad z)                     = 0\quad\text{in}\ \Omega.
\end{alignat*}
This is problem \eqref{adj} with homogeneous data. By Lemma~\ref{la_adj_well}, $(z,\TTheta,\btau)=0$
(where $(0,0,\btau)=0$ is the null element of the quotient space $\U{a}(0)$ when $t=0$).
\end{proof}

\subsection{Proofs of Theorems~\ref{thm_stab},~\ref{thm_DPG}} \label{sec_pf}

We are ready to prove our main results. We start with Theorem~\ref{thm_stab}.
To show the unique and stable solvability of \eqref{VF} it is enough to check the standard properties.
\begin{enumerate}
\item {\bf Boundedness of the functional.} This is immediate since, for $f\in L_2(\Omega)$, it holds
      $L(z)\le \|f\|\,\|z\|\le \|f\|\,\|(z,\TTheta,\btau)\|_{V(\cT,t)}$
      for any $(z,\TTheta,\btau)\in V(\cT,t)$ and $t\in [0,1]$.
\item {\bf Boundedness of the bilinear form.} 
      The bound $b(\uu,\bv)\lesssim \|\uu\|_{U(\cT,t)} \|\bv\|_{V(\cT,t)}$ for all $\uu\in \U{a}(\cT,t)$ and
      $\bv\in V(\cT,t)$ is uniform for $\cT$ and $t\in [0,1]$ due to the selection of norms in
      both spaces.
\item {\bf Injectivity.} In Lemma~\ref{la_inj} we have seen that the adjoint operator
      of $B_t^*:\;V(\cT,t)\to \U{a}(\cT,t)'$ is injective for any $t\in [0,1]$.
\item {\bf Inf-sup condition.} We have to show that
\begin{align} \label{infsup}
   \sup_{0\not=(z,\TTheta,\btau)\in V(\cT,t)}
   \frac {b_t(u,\MM,\btheta,\tq;z,\TTheta,\btau)}{\|z,\TTheta,\btau\|_{V(\cT,t)}}
   \gtrsim \|(u,\MM,\btheta,\tq)\|_{U(\cT,t)} \quad\forall (u,\MM,\btheta,\tq)\in \U{a}(\cT,t)
\end{align}
holds uniformly for $t\in [0,1]$.
As mentioned before, we use the framework from \cite{CarstensenDG_16_BSF}.
For ease of reading let us relate our notation to the one used there:
\begin{align*}
   &X=\U{a}(\cT,t),\quad X_0 = L_2(\Omega)\times\LL_2^s(\Omega)\times\bL_2(\Omega),\quad \hat X=\HRM{a}(\cT,t),\\
   &Y=V(\cT,t),\quad Y_0=\U{a}(t),\quad b(\cdot,\cdot)=b_t(\cdot,\cdot),\\
   &b_0(x,y)=b_t(u,\MM,\btheta,0;z,\TTheta,\btau)\quad
    \text{with $x=(u,\MM,\btheta)$, $y=(z,\TTheta,\btau)$},\\
   &\hat b(\hat x,y)=b_t(0,0,0,\tq;z,\TTheta,\btau)=-\dualRM{\tq}{(z,\TTheta,\btau)}{\cS,t}\quad
    \text{with $\hat x=\tq$, $y=(z,\TTheta,\btau)$}.
\end{align*}
Now, by \cite[Theorem~3.3]{CarstensenDG_16_BSF}, \eqref{infsup} follows from the two inf-sup properties
\begin{align}
   \label{infsup1}
   &\text{\cite[Ass.~3.1]{CarstensenDG_16_BSF}:}\
   \sup_{0\not=(z,\TTheta,\btau)\in \U{a}(t)}
   \frac{b_t(u,\MM,\btheta,0;z,\TTheta,\btau)}{\|(z,\TTheta,\btau)\|_{V(\cT,t)}}
   \gtrsim \|u\| + \|\MM\| + t^{1/2}\|\btheta\|
   \\
   \label{infsup2}
   &\hspace{0.5\textwidth}\forall (u,\MM,\btheta)\in L_2(\Omega)\times\LL_2^s(\Omega)\times\bL_2(\Omega),
   \nonumber\\
   &\text{\cite[(18)]{CarstensenDG_16_BSF}:}\qquad
   \sup_{0\not=(z,\TTheta,\btau)\in V(\cT,t)}
   \frac{\dualRM{\tq}{(z,\TTheta,\btau)}{\cS,t}}{\|(z,\TTheta,\btau)\|_{V(\cT,t)}}
   \gtrsim
   \|\tq\|_{\trRM{\cS,t}} \quad\forall\tq\in\HRM{a}(\cS,t),
\end{align}
and relation
\[
   \U{a}(t) = \{(z,\TTheta,\btau)\in V(\cT,t);\;
              \dualRM{\tq}{(z,\TTheta,\btau)}{\cS,t}=0\ \forall \tq\in \HRM{a}(\cT,t)\}.
\]
This last relation is the statement of Proposition~\ref{prop_jump} (if $t>0$) and
Proposition~\ref{prop_jump0} (if $t=0$).
Inf-sup property \eqref{infsup1} is satisfied due to Lemma~\ref{la_adj_well},
uniformly for $t\in [0,1]$ and subject to the compatibility condition that $\bxi=0$ when $t=0$.
In fact, given $(u,\MM,\btheta)\in L_2(\Omega)\times\LL_2^s(\Omega)\times\bL_2(\Omega)$,
choose $(z^*,\TTheta^*,\btau^*)\in \U{a}(t)$ as the solution of \eqref{adj} with compatible data
$g=u$, $\HH=\MM$ and $\bxi=t^{1/2}\btheta$. Then
\begin{align*}
   \sup_{0\not=(z,\TTheta,\btau)\in \U{a}(t)}
   \frac{b_t(u,\MM,\btheta,0;z,\TTheta,\btau)}{\|(z,\TTheta,\btau)\|_{V(\cT,t)}}
   \ge
   \frac{\|u\|^2 + \|\MM\|^2 + t \|\btheta\|^2}{\|(z^*,\TTheta^*,\btau^*)\|_{U(t)}}
   \gtrsim
   \frac{\|u\|^2 + \|\MM\|^2 + t \|\btheta\|^2}{\|u\| + \|\MM\| + t^{1/2}\|\btheta\|},
\end{align*}
that is, \eqref{infsup1} holds.
Finally, \eqref{infsup2} holds by Proposition~\ref{prop_tr_norms} (with equality and constant $1$).
\end{enumerate}
That $(u,\MM,\btheta)$ satisfies \eqref{prob} and $\tq=\traceRM{\cT,t}(u,\MM,\btheta)$ follows by standard
arguments. This also shows the stated regularity. We have therefore proved Theorem~\ref{thm_stab}.

Recalling that the DPG method delivers the best approximation in the energy norm
$\|\cdot\|_{E(t)}=\|\cdot\|_{V(\cT,t)'}$,
\[
   \|\uu-\uu_h\|_{E(t)} = \min\{\|\uu-\bw\|_{E(t)};\; \bw\in \U{a,h}(\cT,t)\},
\]
to prove Theorem~\ref{thm_DPG}, it is enough to show the uniform equivalence of the energy norm
and the norm $\|\cdot\|_{U(\cT,t)}$.
By definition of the energy norm, $\|\uu\|_{E(t)}\lesssim\|\uu\|_{U(\cT,t)}$ is equivalent
to the boundedness of $b_t(\cdot,\cdot)$, which we have just checked.
The other estimate,  $\|\uu\|_{U(\cT,t)}\lesssim\|\uu\|_{E(t)}$, is the inf-sup property
\eqref{infsup} which also holds. Both estimates hold uniformly for $t\in [0,1]$.

\subsection{Proof of Theorem~\ref{thm_limit}} \label{sec_pf_limit}

By Theorem~\ref{thm_stab} there exists for any $t\in[0,1]$ a unique solution
$(u_t,\MM_t,\btheta_t,\tq_t)\in \U{a}(\cT,t)$ of \eqref{VF}.
Obviously, $\div\Div\MM_t\to\div\Div\MM_0$ in $L_2(\Omega)$ ($t\to 0$)
by \eqref{p1} since $f_t\to f_0$ by assumption and since $\btheta_t=\grad u_t$ due to \eqref{p3}.

Now consider a null sequence of positive numbers $(t_n)$.
By the bound given by Theorem~\ref{thm_stab} and the $L_2$-convergence $f_{t_n}\to f_0$ we have
\begin{align} \label{weak_bound}
   \|u_{t_n}\|^2+\|\MM_{t_n}\|^2 + t_n\|\btheta_{t_n}\|^2
   \le
   \|(u_{t_n},\MM_{t_n},\btheta_{t_n},\tq_{t_n})\|_{U(\cT,t_n)}^2
   \lesssim
   \|f_{t_n}\|^2 \lesssim 1+\|f_0\|^2
\end{align}
for $n$ sufficiently large.
Therefore, there is a subsequence of $(t_n)$, again denoted by $(t_n)$, such that
$(u_{t_n},\MM_{t_n})$ converges weakly to a limit $(u,\MM)\in L_2(\Omega)\times\LL_2^s(\Omega)$.
Note that the symmetry of $\MM$ follows from the symmetry of $\MM_{t_n}$ by testing with skew symmetric tensors
in the weak limit.
Now, selecting $z\in\cD(\Omega)$, $\TTheta\in\DDs(\Omega)$ and $\btau\in\cD(\Omega)$,
it holds $(z,\TTheta,\btau)\in \U{c}(t)\cap\U{s}(t)$ for any $t\in[0,1]$ so that
$\dualRM{\tq_{t_n}}{(z,\TTheta,\btau)}{\cS,t_n}=0$ by Proposition~\ref{prop_jump}.
Thus, formulation \eqref{VF} and the convergence
\begin{align} \label{limit_zero}
   \max\{t_n\|u_{t_n}\|, t_n^2\|\MM_{t_n}\|, t_n\|\btheta_{t_n}\|\} \to 0\quad (n\to\infty)
\end{align}
by \eqref{weak_bound} show that
\begin{align} \label{conv_domain}
   -\vdual{f_{t_n}}{z}
   &=
   b_{t_n}(u_{t_n},\MM_{t_n},\btheta_{t_n},\tq_{t_n};z,\TTheta,\btau)
   \nonumber\\
   &=
   \vdual{u_{t_n}}{\div(\Div\TTheta+t_n(\btau-\grad z)}
   +\vdual{\MM_{t_n}}{\cCinv\TTheta+\Grad(\grad z-t_n^2\Div\TTheta)}
   +t_n\vdual{\btheta_{t_n}}{\btau-\grad z}
   \nonumber\\
   &\to
   \vdual{u}{\div\Div\TTheta}
   +\vdual{\MM}{\cCinv\TTheta+\Grad\grad z}
   \quad (n\to\infty).
\end{align}
Since $\vdual{f_{t_n}}{z}\to \vdual{f_0}{z}$ ($n\to\infty$), it
follows that $\MM\in\HdDiv{\Omega}$, $u\in H^2(\Omega)$ with $-\div\Div\MM=f_0$
and $\MM+\cC\Grad\grad u=0$.

Now, to establish the convergence of $\dualRM{\tq_{t_n}}{(z,\TTheta,\btau)}{\cS,t_n}$, we select
\[
   z\in \cD(\overline\cT):=\{z:\; \Omega\to\R;\; z|_T\in\cD(\bar T)\ \forall T\in\cT\},
\]
$\TTheta\in\DDs(\overline\cT)$, $\btau\in\bD(\overline\cT)$ (with analogous definitions).
Since $\tq_{t_n}=\traceRM{\cS,t_n}(u_{t_n},\MM_{t_n},\btheta_{t_n})$ by Theorem~\ref{thm_stab},
definitions \eqref{trT_def}, \eqref{tr_RM} and the relation $\btheta_{t_n}=\grad u_{t_n}$ show that
\begin{align*}
   \dualRM{\tq_{t_n}}{(z,\TTheta,\btau)}{\cS,t_n}
   &=
   \vdual{u_{t_n}}{\div(\Div\TTheta+t_n(\btau-\grad z)}_\cT
   -\vdual{\div\Div\MM_{t_n}}{z}
   \nonumber\\
   &+\vdual{\MM_{t_n}}{\Grad(\grad z-t_n^2\Div\TTheta)}_\cT
   -\vdual{\Grad(\grad u_{t_n}-t_n^2\Div\MM_{t_n})}{\TTheta}
   +t_n\vdual{\btheta_{t_n}}{\btau-\grad z}_\cT.
\end{align*}
As $(u_{t_n},\MM_{t_n},\btheta_{t_n})$ solves~\eqref{prob}, is holds
$\div\Div\MM_{t_n}=-f_{t_n}$ and $\Grad(\grad u_{t_n} - t_n^2\Div\MM_{t_n})=\cCinv\MM_{t_n}$.
Therefore, the convergence $(u_{t_n},\MM_{t_n})\rightharpoonup (u,\MM)$ in $L_2(\Omega)\times\LL_2^s(\Omega)$
and $f_{t_n}\to f_0$ in $L_2(\Omega)$ together with \eqref{limit_zero} induces the limit
\[
   \dualRM{\tq_{t_n}}{(z,\TTheta,\btau)}{\cS,t_n}
   \to
   \vdual{u}{\div\Div\TTheta}_\cT
   +\vdual{f_0}{z}
   +\vdual{\MM}{\Grad\grad z}_\cT
   -\vdual{\cCinv\MM}{\TTheta} \quad (n\to\infty).
\]
Since $f_0=-\div\Div\MM$ and $\cCinv\MM=\Grad\grad u$, using definitions \eqref{trGG}, \eqref{trDD}
and relation \eqref{trRM0}, this reveals that
\begin{align} \label{conv_trace}
   \dualRM{\tq_{t_n}}{(z,\TTheta,\btau)}{\cS,t_n}
   \to\
   &\dual{\traceGG{}(u)}{\TTheta}_\cS - \dual{\traceDD{}(\MM)}{z}_\cS
   =\dualRM{\traceRM{\cT,0}(u,\MM)}{(z,\TTheta)}{\cS,0}
\end{align}
when $n\to\infty$ so that, arguing as in \eqref{conv_domain},
\begin{align*}
   -\vdual{f_0}{z}
   \leftarrow
   -\vdual{f_{t_n}}{z}
   =
   b_{t_n}(u_{t_n},\MM_{t_n},\btheta_{t_n},\tq_{t_n};z,\TTheta,\btau)
   \to
   b_0(u,\MM,\tq;z,\TTheta)
   \quad (n\to\infty)
\end{align*}
for any $(z,\TTheta,\btau)\in \cD(\overline\cT)\times \DDs(\overline\cT)\times\bD(\overline\cT)$
with $\tq=\traceRM{\cT,0}(u,\MM)$ by \eqref{conv_trace}.
If, for $a=c$, $u\in H^2(\Omega)$ satisfies the homogeneous boundary conditions, i.e., $u\in H^2_0(\Omega)$,
then this means that the limit $(u,\MM,\tq)\in \U{c}(\cT,0)$ solves the Kirchhoff--Love problem of the clamped
plate, \eqref{VF} with $t=0$ and $a=c$, so that $(u_0,\MM_0,\tq_0)=(u,\MM,\tq)$.
On the other hand, if, for $a=s$, $u$ and $\MM$ satisfy the homogeneous boundary conditions
$u\in H^2_s(\Omega)$ and $\MM\in\HdDivz{\Omega}$, then the limit
$(u,\MM,\tq)\in \U{s}(\cT,0)$ solves the Kirchhoff--Love problem \eqref{VF} with $t=0$ and $a=s$.

It therefore remains to show the corresponding homogeneous boundary conditions.

{\bf Case $a=c$.}
Selecting $z=0$, $\btau=0$ and $\TTheta\in\DDs(\bar\Omega)$,
the boundary conditions $u_{t_n}=0$, $\grad u_{t_n}-t_n^2\Div\MM_{t_n}=0$ on $\Gamma$, cf.~\eqref{pBCc},
Lemma~\ref{la_tr} and the weak convergence \eqref{conv_trace} show that
\[
   \dualRM{\tq_{t_n}}{(z,\TTheta,\btau)}{\cS,t_n}
   =
   -t_n^2\dual{\MM_{t_n}\nn}{\Div\TTheta}_\Gamma
   \to
   \dual{\traceGG{}(u)}{\TTheta}_\cS\quad (n\to\infty).
\]
On the other hand,
\begin{align*}
   t_n^2\dual{\MM_{t_n}\nn}{\Div\TTheta}_\Gamma
   =
   t_n^2\vdual{\Div\MM_{t_n}}{\Div\Theta}+t_n^2\vdual{\MM_{t_n}}{\Grad\Div\TTheta}
   \to 0\quad (n\to\infty)
\end{align*}
since $\|\MM_{t_n}\|\lesssim 1+\|f_0\|$ (used in \eqref{limit_zero}) and
$t_n\|\Div\MM_{t_n}\|\lesssim 1+\|f_0\|$ for $n$ sufficiently large,
so that $\dual{\traceGG{}(u)}{\TTheta}_\cS=0$. Indeed, by Corollary~\ref{cor_adj_well},
\begin{align*}
  t_n^2 &\|\Div\MM_{t_n}\|^2
  \le
  \|(u_{t_n},\MM_{t_n},\btheta_{t_n})\|_{U(t)}^2
  \\
  &=
  \|u_{t_n}\|^2 + 2t\|\btheta_{t_n}\|^2 + \|\MM_{t_n}\|^2
  + \|\Grad(\grad u_{t_n}-t_n^2\Div\MM_{t_n})\|^2
  + \|f_{t_n}\|^2
  \lesssim
  \|f_{t_n}\|^2 \lesssim 1 + \|f_0\|^2.
\end{align*}
Here, we used relations \eqref{p1}, \eqref{p2}, \eqref{p3} and bound \eqref{weak_bound}.
Using Lemma~\ref{la_trRM0_asym}, specifically relation \eqref{trGGDD},
we conclude that $\dual{\traceGG{}(u)}{\TTheta}_\cS=\dual{\traceDD{}(\TTheta)}{u}_\cS=0$
for any $\TTheta\in\DDs(\bar\Omega)$ so that $u\in H^2_0(\Omega)$ by \cite[Proposition 3.8(i)]{FuehrerHN_19_UFK}.

{\bf Case $a=s$.}
First we show that $u\in H^2_s(\Omega)$. We select $z=0$, $\btau=0$ and
$\TTheta\in\HDiv{\Omega}\cap\HdDivz{\Omega}$.
Then, similarly as before, we conclude that
\[
    \dualRM{\tq_{t_n}}{(z,\TTheta,\btau)}{\cS,t_n}
    =
    \dual{u_{t_n}}{\nn\cdot\Div\TTheta}_\Gamma - t_n^2 \dual{\MM_{t_n}\nn}{\Div\TTheta}_\Gamma
    = 0
    \to
    \dual{\traceGG{}(u)}{\TTheta}_\cS
\]
when $n\to\infty$ since $u_{t_n}=0$ and $\MM_{t_n}\nn=0$ on $\Gamma$.
In other words, $u\in H^2_s(\Omega)$, by Corollary~\ref{cor_jump0} and the density
of $\HDiv{\Omega}\cap\HdDivz{\Omega}$ in $\HdDivz{\Omega}$.

Now, to show that $\MM\in\HdDivz{\Omega}$, we select $z\in H^2_s(\Omega)$ and $\btau=0$, $\TTheta=0$,
and use that $u_{t_n}=0$, $\MM_{t_n}\nn=0$ on $\Gamma$, cf.~\eqref{pBCs},
and $\btheta_{t_n}=\grad u_{t_n}$. Then Lemma~\ref{la_tr} and the weak convergence \eqref{conv_trace}
imply that
\[
   \dualRM{\tq_{t_n}}{(z,\TTheta,\btau)}{\cS,t_n}
   =
   -\dual{\nn\cdot\Div\MM_{t_n}}{z}_\Gamma
   =
   0
   \to
   -\dual{\traceDD{}(\MM)}{z}_\cS\quad (n\to\infty).
\]
It follows that $\MM\in\HdDivz{\Omega}$, cf.~\eqref{HdDivz_def}.

Finally, since the sequence $(t_n)$ was arbitrary, we have established the weak convergence
of the Reissner--Mindlin solution to the Kirchhoff--Love solution, for the boundary
conditions of the clamped plate and the simply supported plate.
\input{examples}

\bibliographystyle{siam}
\bibliography{/home/norbert/tex/bib/heuer,/home/norbert/tex/bib/bib}
\end{document}

%% file: header.tex
\newtheorem{theorem}{Theorem}
\newtheorem{lemma}[theorem]{Lemma}
\newtheorem{cor}[theorem]{Corollary}
\newtheorem{prop}[theorem]{Proposition}
\newtheorem{remark}[theorem]{Remark}
\theoremstyle{definition} 

\newcommand{\<}{\langle{}}
\renewcommand{\>}{\rangle}

\newcommand{\ip}[2]{\llangle#1\hspace*{.5mm},#2\rrangle}
\newcommand{\dual}[2]{\<#1\hspace*{.5mm},#2\>}
\newcommand{\dualRM}[3]{\<#1\hspace*{.5mm},#2\>_{#3}}
\newcommand{\vdual}[2]{(#1\hspace*{.5mm},#2)}

\newcommand{\wilde}{\widetilde}
\newcommand{\wat}{\widehat}
\newcommand{\transp}{\mathsf{T}}

\def\Grad{\boldsymbol{\varepsilon}}

\def\Div{{\rm\bf div\,}}
\def\pwDiv{{\rm\bf div}_\cT}
\def\grad{\nabla}
\def\pwgrad{\nabla_\cT}
\def\MM{\mathbf{M}}
\def\II{\mathbf{I}}

\def\TTheta{\mathbf{\Theta}}
\def\bq{\boldsymbol{q}}

\def\rr{\mathbf{r}}

\newcommand{\bL}{\ensuremath{\mathbf{L}}}
\newcommand{\LL}{\ensuremath{\mathbb{L}}}

\def\tq{\wat{\boldsymbol{q}}}

\def\btheta{\boldsymbol{\theta}}

\def\tv{\wat{\boldsymbol{v}}}

\newcommand{\bg}{\ensuremath{\mathbf{g}}}

\newcommand{\bu}{\ensuremath{\mathbf{u}}}

\newcommand{\vv}{\boldsymbol{v}}
\newcommand{\bv}{\ensuremath{\mathbf{v}}}
\newcommand{\uu}{\ensuremath{\mathbf{u}}}

\newcommand{\bw}{\ensuremath{\mathbf{w}}}

\newcommand{\brho}{\ensuremath{\boldsymbol{\rho}}}
\newcommand{\bpsi}{\ensuremath{\boldsymbol{\psi}}}

\newcommand{\bxi}{\ensuremath{\boldsymbol{\xi}}}

\newcommand{\deltaz}{\delta\!z}
\newcommand{\DeltaTheta}{\boldsymbol{\delta}\!\TTheta}
\newcommand{\deltatau}{\boldsymbol{\delta}\!\btau}

\newcommand{\deltaq}{\boldsymbol{\delta}\!\tq}
\newcommand{\bdeltau}{\boldsymbol{\delta}\!\bu}

\newcommand{\traceRM}[1]{\mathrm{tr}_{#1}^{\mathrm{RM}}}
\newcommand{\traceDD}[1]{\mathrm{tr}_{#1}^{\mathrm{dDiv}}} 
\newcommand{\traceGG}[1]{\mathrm{tr}_{#1}^{\mathrm{Ggrad}}} 

\newcommand{\dDiv}{{\div\Div\!}}
\newcommand{\Hdiv}[1]{{H(\div\!,#1)}}
\newcommand{\HDiv}[1]{{\bH(\Div\!,#1)}}
\newcommand{\HDivz}[1]{{\bH_0(\Div\!,#1)}}
\newcommand{\HdDiv}[1]{{H(\dDiv,#1)}}
\newcommand{\HdDivz}[1]{{H_0(\dDiv,#1)}}

\newcommand{\U}[1]{U_{#1}}
\newcommand{\HRM}[1]{\ensuremath{\mathbf{H}^{\mathrm{RM}}_{#1}}}
\newcommand{\HH}{\ensuremath{\mathbf{H}}}
\newcommand{\bH}{\ensuremath{\mathbf{H}}}
\newcommand{\DDs}{\ensuremath{{\mathbb{D}}^s}}
\newcommand{\bD}{\ensuremath{\boldsymbol{\mathcal{D}}}}
\newcommand{\cD}{\ensuremath{\mathcal{D}}}

\newcommand{\trggrad}[1]{{\mathrm{Ggrad},#1}}
\newcommand{\trddiv}[1]{{\mathrm{dDiv},#1}}
\newcommand{\trRM}[1]{{\mathrm{RM},#1}}

\def\div{{\rm div\,}}

\newcommand{\ttt}{{\rm T}}

\newcommand{\di}{d}
\newcommand{\R}{\ensuremath{\mathbb{R}}}

\newcommand{\nn}{\ensuremath{\mathbf{n}}}

\newcommand{\cC}{\ensuremath{\mathcal{C}}}
\newcommand{\cCinv}{\ensuremath{\mathcal{C}^{-1}}}
\newcommand{\cT}{\ensuremath{\mathcal{T}}}
\newcommand{\cO}{\ensuremath{\mathcal{O}}}

\newcommand{\cS}{\ensuremath{\mathcal{S}}}

\newcommand{\btau}{{\boldsymbol\tau}}


%% file: examples.tex
\section{Numerical experiment} \label{sec_num}
In this section we study a simple model problem with smooth solutions (depending on $t$).
As mentioned before, a fully discrete analysis (taking an approximation of optimal test
functions into account) is an open subject. Also the construction of low-regular
basis functions for the discretization of trace spaces is ongoing research.
Here, we are only interested to investigate robustness of our scheme with respect to the parameter $t>0$.

Our constructed model problem is as follows. We consider a plate with mid-surface
$\Omega = (0,1)^2$ and select $\cC$ as the identity.
Given the (rescaled) rotation vector
\begin{align*}
  \bpsi(x,y) = \begin{pmatrix}
    y^3(1-y)^3 x^2(1-x)^2 (2x-1) \\
    x^3(1-x)^3 y^2(1-y)^2 (2y-1)
  \end{pmatrix}
\end{align*}
we set $\MM := -\Grad\bpsi$ and select the (rescaled) bending load $f:=-\dDiv\MM$.
The deflection $u\in H^2(\Omega)\cap H_0^1(\Omega)=H^2_s(\Omega)$ can then be obtained from relation
\(
  \grad u = \bpsi + t^2 \Div\MM.
\)
Note that $f$ and $\MM$ are independent of the thickness parameter $t$ whereas the deflection $u$ depends
on this parameter.
Furthermore, one verifies that the solution $u$ of this problem satisfies
the clamped plate boundary conditions \eqref{pBCc} as well as the boundary conditions \eqref{pBCs}
of the simply supported plate.
In the example presented here we only consider the latter pair of boundary conditions, \eqref{pBCs},
that is, $a=s$ in the setting of our spaces.

Recall from Section~\ref{sec_uw} the ansatz space
\begin{align*}
  \U{s}(\cT,t) := L_2(\Omega)\times\LL_2^s(\Omega)\times\bL_2(\Omega)\times \HRM{s}(\cS,t).
\end{align*}
We replace the spaces for the $L^2(\Omega)$ field variables $u,\MM,\btheta$ by spaces of element-wise constant functions,
i.e., $L_2(\Omega)\times\LL_2^s(\Omega)\times\bL_2(\Omega)$ is replaced by
\begin{align*}
  P^0(\cT) \times P^0(\cT)^{2\times 2}\cap \LL_2^s(\Omega) \times P^0(\cT)^2,
\end{align*}
where $P^p(\cT)$ denotes the space of element-wise polynomials of degree $\leq p$.
Here, we use a triangulation $\cT$ of the computational domain where the initial mesh contains four elements.
For the choice of an appropriate approximation space of the traces we utilize the fact that the exact solution $u$ and
$\MM$ are regular, $u\in H^2(\Omega)$ and $\MM \in H^2(\Omega)^{2\times 2} \cap \LL_2^s(\Omega)$.
Let $U_h \subset H^2_s(\Omega)$ denote the space of reduced HCT-elements.
We note that traces of this space have also been
used in our previous works to discretize the ultraweak formulation of the Kirchhoff--Love model problem,
see~\cite{FuehrerHN_19_UFK} and~\cite{FuehrerH_19_FDD}.
We then define the space
\begin{align*}
  \HRM{s,h}(\cS,t) := \{\traceRM{T,t}(u_h,\MM_h,\btheta_h):\; u_h\in U_h,\, \MM_h\in U_h^{2\times 2}\cap
  \LL_2^s(\Omega),\, \btheta_h = \grad u_h\}.
\end{align*}
Elements of this space satisfy the boundary conditions \eqref{pBCs} of the simply supported plate,
that is, $\HRM{s,h}(\cS,t)\subset \HRM{s}(\cS,t)$.

Instead of using optimal test functions in the space $V(\cT,t)$, we consider the finite dimensional space
\begin{align*}
  V_h(\cT,t) = P^3(\cT) \times P^3(\cT)^{2\times 2} \times P^3(\cT)^2
\end{align*}
and use an approximated trial-to-test operator by replacing $V(\cT,t)$ in \eqref{ttt} with $V_h(\cT,t)$.

\begin{figure}
  \begin{center}
    \includegraphics{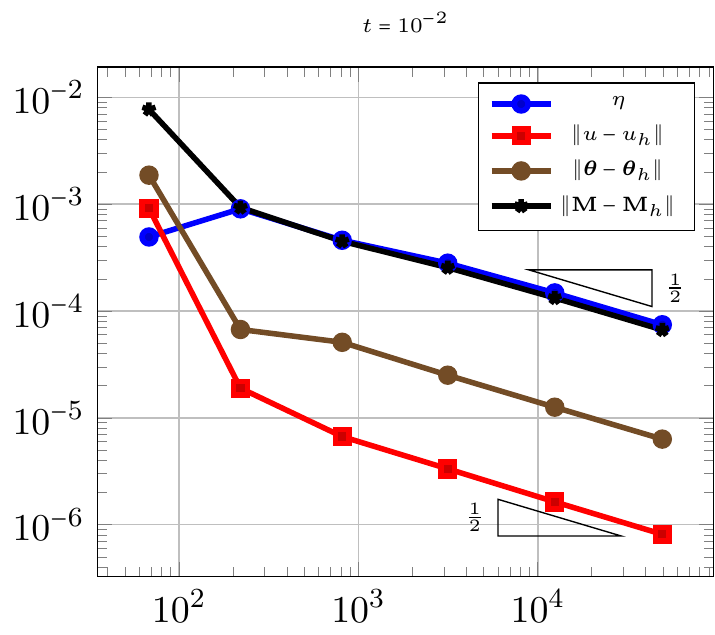}
    \includegraphics{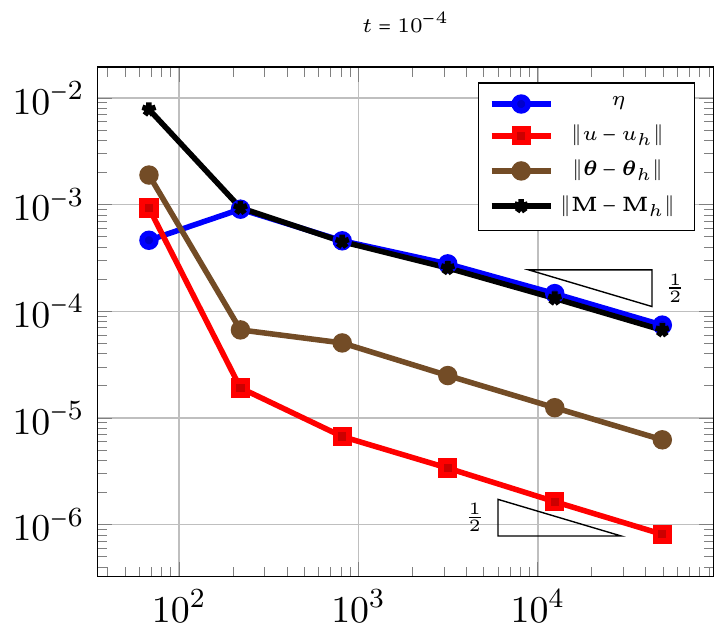}
    \includegraphics{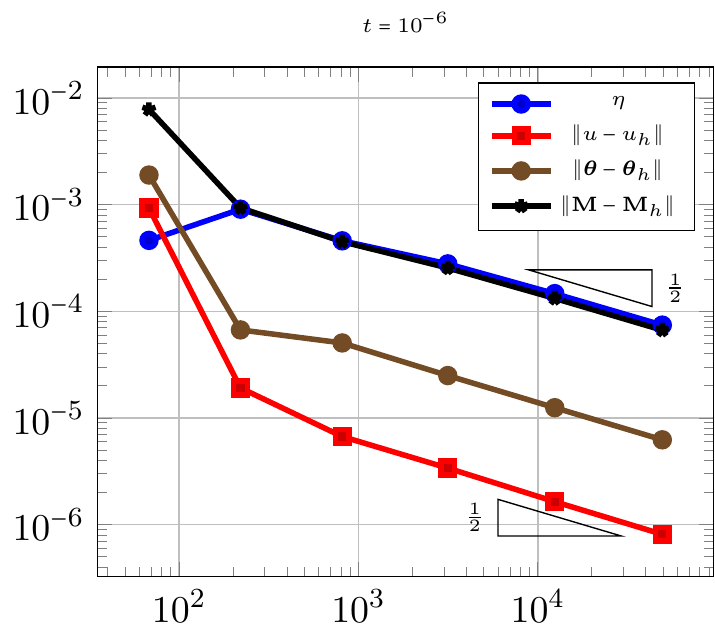}
    \includegraphics{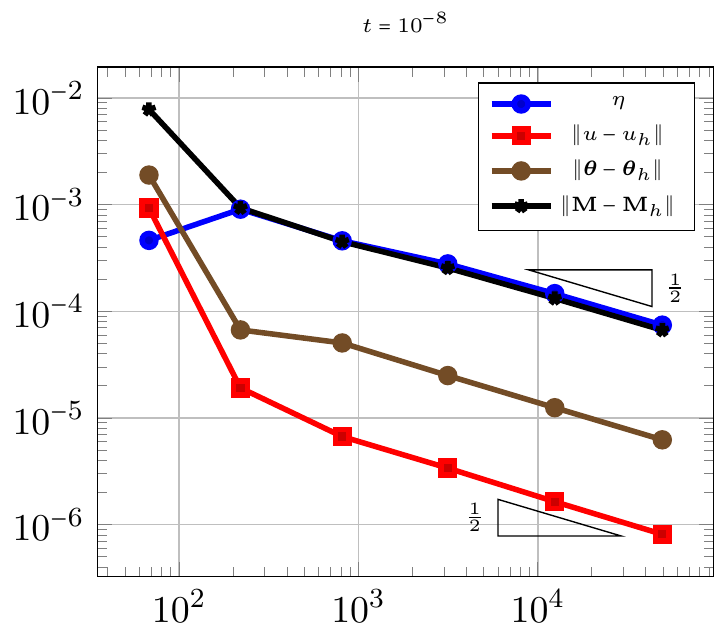}
  \end{center}
  \caption{Errors of the field variables in the $L_2(\Omega)$-norm and estimator $\eta$ versus
           the number of degrees of freedom for $t=10^{-j}$, $j\in\{2,4,6,8\}$.}
  \label{fig_t}
\end{figure}

We perform numerical experiments with a sequence of uniformly refined meshes, and for different
values of $t$ ($t=10^{-j}$, $j\in\{2,4,6,8\}$). Figure~\ref{fig_t} shows the errors of the field variables
$\|u-u_h\|$, $\|\MM-\MM_h\|$, $\|\btheta-\btheta_h\|$ ($\uu_h=(u_h,\MM_h,\btheta_h,\tq_h)$ being the
DPG approximation of $\uu=(u,\MM,\btheta,\tq)$) along with the DPG estimator
\begin{align*}
  \eta = \sup_{0\not=\vv_h=(v_h,\TTheta_h,\btau_h)\in V_h(\cT,t)}
  \frac{b_t(\uu_h;\vv_h)-\vdual{f}{v_h}}{\|\vv_h\|_{V(\cT,t)}}.
\end{align*}
This estimator is an approximation to the error of the residual $\|B_t\uu_h-L\|_{V(\cT,t)'}$,
cf.~\eqref{DPG} and the discussion there.
We observe that $\eta$ is an upper bound for the total error in the field variables, as expected.
Furthermore, the error curves are almost independent of $t$, thus confirming our error estimates
which are uniform in $t$.
We also observe that the error of $\btheta_h$ seems to be controlled in a stronger norm (without $t$-weighting).